\documentclass[12pt]{article}

\usepackage{amsmath}
\usepackage{amsthm}
\usepackage{amssymb}
\usepackage{cite}
\usepackage{url}
\usepackage[multiple]{footmisc}
\usepackage{graphicx}
\usepackage{epstopdf}
\usepackage{chngcntr}
\usepackage{enumitem}
\usepackage{hhline}
\usepackage{array}
\usepackage{hyperref}
\usepackage{enumitem}
\usepackage{color}
\usepackage{subcaption}
\usepackage{float}
\usepackage{capt-of}

\hypersetup{colorlinks=false}

\if@twoside
\else 
\fi 

\numberwithin{equation}{section}
\counterwithin{figure}{section}
\counterwithin{table}{section}

\theoremstyle{plain}
\newtheorem{theorem}{Theorem}[section]
\newtheorem{lemma}[theorem]{Lemma}
\newtheorem{proposition}[theorem]{Proposition}
\newtheorem{corollary}[theorem]{Corollary}
\newtheorem*{problem}{Problem}

\theoremstyle{definition}
\newtheorem{definition}{Definition}[section]

\newtheorem{example}{Example}[section]
\newtheorem*{notation}{Notation}
\newtheorem{assumption}{Assumption}[section]

\theoremstyle{remark}
\newtheorem*{remark}{Remark}

\newcommand{\norm}[1]{\left\|#1\right\|}
\newcommand{\abs}[1]{\left\vert#1\right\vert}
\newcommand{\spr}[1]{\left\langle\,#1\,\right\rangle}
\newcommand{\kl}[1]{\left(#1\right)}
\newcommand{\vl}{\, \vert \,}

\newcommand{\supp}{\operatorname{supp}}

\newcommand{\intV}[1]{\int\limits_\Omega #1 \, dx}
\newcommand{\intVo}[1]{\int\limits_\Oo #1 \, dx}

\newcommand{\meas}[1]{\operatorname{meas}\kl{#1}}

\newcommand{\sprD}[1]{\left\langle \, #1 \, \right\rangle_{\Vs,V}}

\newcommand{\sprf}[1]{\left\langle \, #1 \, \right\rangle_{\HmoO,\HoO}}
\newcommand{\sprgT}[1]{\left\langle \, #1 \, \right\rangle_{\HmohObT,\HohbT}}

\newcommand{\normDsD}[1]{\left\|#1\right\|_{\Vs,V}}
\newcommand{\normDDs}[1]{\left\|#1\right\|_{V,\Vs}}
\newcommand{\normHDs}[1]{\left\|#1\right\|_{\HoO,\Vs}}

\newcommand{\grad}{\nabla}

\renewcommand{\div}[1]{\operatorname{div}\left(#1\right)}
\newcommand{\mE}[1]{\mathcal{E}\left(#1\right)}

\newcommand{\D}{\mathcal{D}}

\newcommand{\Range}{\mathcal{R}}

\newcommand{\Ds}{\mathcal{D}_s}

\newcommand{\HsO}{{H^s(\Omega)}}

\newcommand{\HoO}{{H^1(\Omega)}}
\newcommand{\HohbT}{{H^{\frac{1}{2}}(\GT)}}
\newcommand{\HohbD}{{H^{\frac{1}{2}}(\GD)}}
\newcommand{\HtO}{{H^2(\Omega)}}

\newcommand{\HtOo}{{H^2(\Oo)}}
\newcommand{\HmoO}{{H^{-1}(\Omega)}}
\newcommand{\HmohObT}{{H^{-\frac{1}{2}}(\Gamma_T)}}

\newcommand{\LoO}{{L^1(\Omega)}}
\newcommand{\LtO}{{L^2(\Omega)}}

\newcommand{\LtOo}{{L^2(\Oo)}}

\newcommand{\LiO}{{L^\infty(\Omega)}}

\newcommand{\xD}{x^\dagger}

\newcommand{\xkd}{x_k^\delta}
\newcommand{\xkmd}{x_{k-1}^\delta}
\newcommand{\xkpd}{x_{k+1}^\delta}

\newcommand{\xksd}{x_{k_*}^\delta}

\newcommand{\zkd}{z_{k}^\delta}

\newcommand{\skd}{s_k^\delta}

\newcommand{\akd}{\alpha_k^\delta}

\newcommand{\okd}{{\omega_k^\delta}}

\newcommand{\rb}{\bar{\rho}}

\newcommand{\Btr}{\mathcal{B}_{2\rho}}

\newcommand{\Brb}{\mathcal{B}_{\bar{\rho}} }

\newcommand{\cK}{c_K}
\newcommand{\cT}{c_T}
\newcommand{\cF}{c_F}

\newcommand{\cEs}{c_E^s}
\newcommand{\cLM}{c_{LM}}
\newcommand{\cNL}{c_{NL}}
\newcommand{\cR}{c_{R}}

\newcommand{\cG}{c_{G}}

\newcommand{\cP}{c_{P}}

\newcommand{\mb}{\underline{\mu}}

\newcommand{\mub}{\bar{\mu}}
\newcommand{\lb}{\bar{\lambda}}

\newcommand{\ub}{\bar{u}}

\newcommand{\uh}{\hat{u}}

\newcommand{\ut}{\tilde{u}}

\newcommand{\At}{\tilde{A}}

\newcommand{\lh}{\hat{\lambda}}
\newcommand{\muh}{\hat{\mu}}

\newcommand{\Oo}{{\Omega_1}}

\newcommand{\E}{\mathcal{E}}

\newcommand{\ks}{{k_*}}

\newcommand{\GD}{{\Gamma_D}}
\newcommand{\GT}{{\Gamma_T}}

\newcommand{\bO}{{\partial \Omega}}

\newcommand{\eps}{\varepsilon}

\newcommand{\fot}{\frac{1}{t}}
\newcommand{\lkd}{\lambda^\delta_k}

\newcommand{\vs}{v^*}
\newcommand{\vsb}{\bar{v}^*}

\newcommand{\vb}{\bar{v}}

\newcommand{\Brh}{\mathcal{B}_{r_1,r_2}^{h_1,h_2}}
\newcommand{\Srh}{\mathcal{S}_{r_1,r_2}^{h_1,h_2}}

\newcommand{\Mmb}{\mathcal{M}(\mb)}

\newcommand{\hh}{h_\lambda,h_\mu}

\newcommand{\lD}{{\lambda^\dagger}}
\newcommand{\mD}{{\mu^\dagger}}
\newcommand{\lm}{{\lambda,\mu}}
\newcommand{\lmD}{{\lambda^\dagger,\mu^\dagger}}
\newcommand{\lmz}{{\lambda_0,\mu_0}}

\newcommand{\lmth}{{\lambda +  t h_\lambda,\mu +t h_\mu}}

\newcommand{\lmb}{{\bar{\lambda},\bar{\mu}}}
\newcommand{\lmbd}{{\bar{\lambda}- \lambda, \bar{\mu}-\mu}}

\newcommand{\lDz}{{\lambda^\dagger - \lambda_0}}
\newcommand{\mDz}{{\mu^\dagger - \mu_0}}

\newcommand{\alm}{a_{\lm}}

\newcommand{\almb}{a_{\lmb}}
\newcommand{\ahh}{a_{\hh}}

\newcommand{\almbd}{a_{\lmbd}}

\newcommand{\Alm}{A_{\lm}}
\newcommand{\AlmD}{A_{\lmD}}
\newcommand{\Almth}{A_{\lmth}}

\newcommand{\Almb}{A_{\lmb}}

\newcommand{\Almbd}{A_{\lmbd}}
\newcommand{\Ahh}{A_{\hh}}
\newcommand{\Atlm}{\tilde{A}_{\lm}}

\newcommand{\Atlmb}{\tilde{A}_{\lmb}}
\newcommand{\Atlmth}{\tilde{A}_{\lmth}}

\newcommand{\Atlmbd}{\tilde{A}_{\lmbd}}
\newcommand{\Athh}{\tilde{A}_{\hh}}

\newcommand{\Glm}{{G_{\lm}}}

\newcommand{\ud}{u^\delta}
\newcommand{\ulm}{u(\lm)}
\newcommand{\ulmth}{u(\lmth)}

\newcommand{\gD}{g_D}
\newcommand{\gT}{g_T}

\newcommand{\Vs}{{V^*}}

\newcommand{\HoS}{{H^{1}_{0,\Gamma_D}(\Omega)}}

\newcommand{\WoiO}{{W^{1,\infty}(\Omega)}}

\newcommand{\WoiOo}{{W^{1,\infty}(\Oo)}}

\newcommand{\n}{\vec{n}}

\newcommand{\Fc}{F_c}

\newcommand{\Es}{{E_s}}

\newcommand{\vt}{\tilde{v}}

\newcommand{\Frs}{F\vert_{\Ds(F)}}

\title{Lam\'e Parameter Estimation from Static Displacement Field Measurements in the Framework of Nonlinear Inverse Problems}
\author{
Simon Hubmer\footnote{Johannes Kepler University Linz, Doctoral Program Computational Mathematics, Altenbergerstra{\ss}e 69, A-4040 Linz, Austria (simon.hubmer@dk-compmath.jku.at), corresponding author.},
Ekaterina Sherina\footnote{Technical University of Denmark, Department of Applied Mathematics and Computer Science, Asmussens All\'e, 2800 Kongens Lyngby, Denmark (sershe@dtu.dk)},
Andreas Neubauer\footnote{Johannes Kepler University Linz, Industrial Mathematics Institute, Altenbergerstra{\ss}e 69,
A-4040 Linz, Austria (neubauer@indmath.uni-linz.ac.at)},
Otmar Scherzer\footnote{University of Vienna, Computational Science Center, Oskar Morgenstern-Platz 1, 1090 Vienna, Austria (otmar.scherzer@univie.ac.at)} \footnote{Johann Radon Institute Linz, Altenbergerstra{\ss}e 69, A-4040 Linz, Austria (otmar.scherzer@univie.ac.at)}
}

\begin{document}

\maketitle

\begin{abstract}
We consider a problem of quantitative static elastography, the estimation of the Lam\'e parameters from internal 
displacement field data. This problem is formulated as a nonlinear operator equation. To solve 
this equation, we investigate the Landweber iteration both analytically and numerically.
The main result of this paper is the verification of a nonlinearity condition in an \emph{infinite dimensional} 
Hilbert space context. This condition guarantees convergence of iterative regularization methods.
Furthermore, numerical examples for recovery of the Lam\'e parameters from displacement data 
simulating a static elastography experiment are presented.

\medskip
\noindent \textbf{Keywords:} Elastography, Inverse Problems, Nonlinearity Condition, 
Linearized Elasticity, Lam\'e Parameters, Parameter Identification, Landweber Iteration

\medskip
\noindent \textbf{AMS:} 65J22, 65J15, 74G75
\end{abstract}

\section{Introduction}
Elastography is a common technique for medical diagnosis. Elastography can be implemented based on any 
imaging technique by recording successive images and evaluating the displacement data (see  
\cite{Ophir_Cespedes_Ponnekanti_Yazdi_Li_1991, ODonnel_Skovoroda_Shapo_Emelianov_1994, Sarvazyan_Skovoroda_Emelianov_Fowlkes_Pipe_Adler_Buxton_Carson_1995, Lubinski_Emelianov_ODonnel_1999}, which are some early references 
on elastographic imaging based on ultrasound imaging). 
We differ between \emph{standard elastography}, which consists in displaying the displacement data, and 
\emph{quantitative elastography}, which consists in reconstructing elastic material parameters. 
Again we differ between two kinds of inverse problems related to quantitative elastography: 
The \emph{all in once approach} attempts to estimate the elastic material parameters from direct measurements of the underlying 
imaging system (typically recorded outside of the object of interest), 
while the \emph{two-step approach} consists in successive tomographic imaging, displacement computation and 
quantitative reconstruction of the elastic parameters from \emph{internal data}, which is computed from reconstructions of 
a tomographic imaging process. 
The fundamental difference between these approaches can be seen by a dimensionality analysis: Assuming that the 
material parameter is isotropic, it is a scalar locally varying parameter in three space dimensions. Therefore, three dimensional 
measurements of the imaging system should be sufficient to reconstruct the material parameter. On the other hand, the 
displacement data are a three-dimensional vector field, which requires ``three times as much information''. The second approach is 
more intuitive, but less data economic, since it builds up on the well-established reconstruction process taking into account 
the image formation process, 
and it can be implemented successfully if appropriate prior information can be used, such as smoothness assumptions or significant 
speckle for accurate tracking. In this paper we follow the second approach.

In this paper we assume that the model of \emph{linearized elasticity}, describing the relation between 
forces and displacements, is valid. Then, the inverse problem of \emph{quantitative elastography with internal measurements} 
consists in estimating the spatially varying \emph{Lam\'e parameters} $\lm$ from displacement field measurements $u$ induced by 
external forces.  

There exist a vast amount of mathematical literature on identifiability of the Lam\'e parameters, 
stability, and different reconstruction methods. See for example 
\cite{Bal_Bellis_Imperiale_2014,Bal_Uhlmann_2012,Bal_Uhlmann_2013,Barbone_Gokhale_2004,Barbone_Oberai_2007,Doyley_2012,Doyley_Meaney_Bamber_2000,Fehrenbach_Masmoudi_Souchon_2006,Gokhale_Barbone_Oberai_2008,Huang_Shih_1997,Jadamba_Khan_Raciti_2008,Ji_McLaughlin_Renzi_2003,McLaughlin_Renzi_2006,Oberai_Gokhale_Doyley_2004,Oberai_Gokhale_Feijoo_2003,Widlak_Scherzer_2015,Lechleiter_Schlasche_2017, Kirsch_Rieder_2016} and the references therein. Many of the above papers deal with the time-dependent equations of linearized elasticity, since the resulting inverse problem is arguably more stable because it uses more data. However, in many applications, including the ones we have in mind, no dynamic, i.e., time-dependent displacement field data, are available and hence one has to work with the static elasticity equations.

In this paper we consider the inverse problem of identifying the Lam\'e parameters from \emph{static} 
displacement field measurements $u$. We reformulate this problem as a nonlinear operator equation 
	\begin{equation}\label{prob_main}
 		F(\lm) = u \,,
	\end{equation}
in an \emph{infinite dimensional} Hilbert space setting, 
which enables us to solve this equation by gradient based algorithms. In particular, we 
are studying the convergence of the Landweber iteration, which can be considered a gradient descent algorithm (without line search)
in an infinite dimensional function space setting, and reads as follows:
	\begin{equation}\label{eq:LW}
 		(\lambda^{(k+1)},\mu^{(k+1)}) = (\lambda^{(k)},\mu^{(k)}) - (F'(\lambda^{(k)},\mu^{(k)}))^* (F(\lambda^{(k)},\mu^{(k)})-u^\delta)\,,
	\end{equation}
where $k$ is the iteration index of the Landweber iteration. The iteration is terminated when for the first time $\|F(\lambda^{(k)},\mu^{(k)})-u^\delta\| < \tau \delta$, where $\tau > 1$ is a constant and $\delta$ is an estimate for the amount of noise in the data $u^\delta \approx u$.
Denoting the termination index by $k_*:=k_*(\delta)$, and assuming a nonlinearity condition on $F$ to hold, guarantees that 
$(\lambda^{(k_*-1)},\mu^{(k_*-1)})$ approximates the desired solution of \eqref{prob_main} (that is, it is convergent in the 
case of noise free data), and for $\delta \to 0$, $(\lambda^{(k_*(\delta)-1)},\mu^{(k_*(\delta)-1)})$ is continuously depending 
on $\delta$ (that is, the method is stable \cite{Kaltenbacher_Neubauer_Scherzer_2008}). The main ingredient in the 
analysis is a non-standard nonlinearity condition, called the \emph{tangential cone condition}, in an \emph{infinite dimensional} 
functional space setting, which is verified in Section \ref{sec:nc}. The \emph{tangential cone condition} has been subject 
to several studies for particular examples of inverse problems (see for instance \cite{Kaltenbacher_Neubauer_Scherzer_2008}). 
In infinite dimensional function space settings it has only been verified for very simple test case, while after discretization 
it can be considered a consequence of the inverse function theorem. This condition has been verified for instance for the 
\emph{discretized} electrical impedance tomography problem \cite{Lechleiter_Rieder_2008}. 
The motivation for studying the Landweber iteration in an infinite dimensional setting is that the 
convergence is discretization independent, and when actually discretized for numerical purposes, no additional 
discretization artifacts appear. That means that the outcome of the iterative algorithm after stopping by a 
discrepancy principle is approximating the desired solution of \eqref{prob_main} and is also stable with respect 
to data perturbations in an \emph{infinite dimensional setting}. 
However, stability estimates, such as \cite{Lai_2014}, cannot be derived from this condition alone, but follow 
if source conditions, like \eqref{sourcecond_general}, are satisfied (see \cite{Scherzer_2001}). For \emph{dynamic} measurement data of the displacement field $u$, related investigation have been performed in 
\cite{Lechleiter_Schlasche_2017,Kirsch_Rieder_2016}.

The outline of this paper is as follows: First, we recall the equations of linear elasticity, describing the 
forward model (Section \ref{sec:MLE}). Then, we calculate the Fr\`echet derivative and its adjoint (Sections \ref{sec:frechet} and 
\ref{sec:frechet_adjoint}), which are needed to implement the Landweber iteration.
The main result of this paper is the verification of the \emph{(strong) nonlinearity condition} (Section  \ref{sec:nc}) from 
\cite{Hanke_Neubauer_Scherzer_1995} 
in an \emph{infinite dimensional} setting, which is the basic assumption guaranteeing convergence of iterative regularization methods. 
Therefore, together with the general convergence rates results from \cite{Hanke_Neubauer_Scherzer_1995} our paper provides the first 
successful convergence analysis (guaranteeing convergence to a minimum energy solution) of an iterative method for quantitative elastography in a function space setting.
Finally, we present some sample reconstructions with iterative regularization methods from numerically simulated displacement 
field data (Section \ref{sec:source}).

\section{Mathematical Model of Linearized Elasticity}
\label{sec:MLE}

In this section we introduce the basic notation and recall the basic equation of linearized elasticity:
\begin{notation}
 $\Omega$ denotes a non-empty bounded, open and connected set in $\mathbb{R}^N$, $N=1,2,3$, with a Lipschitz continuous 
boundary $\bO$, which has two subsets $\GD$ and $\GT$, satisfying
$\bO = \overline{\GD \cup \GT}$, $\GD \cap \GT = \emptyset$ and $\meas{\GD} > 0$.
\end{notation}
\begin{definition}
Given body forces $f$, displacement $g_D$, surface traction $g_T$ and Lam\'e parameters $\lambda$ and $\mu$, 
the forward problem of linearized elasticity with displacement-traction boundary conditions consists in finding 
$\ut$ satisfying
  \begin{equation} \label{prob_forward_non_hom}
  \begin{split}
      -\div{\sigma(\ut)} & = f \,, \quad \text{in} \; \Omega \,, \\
      \ut \,|_{\GD} & = g_D \,, \\
      \sigma(\ut) \n \,|_{\GT} & = g_T \,,
  \end{split}
  \end{equation}
where $\n$ is an outward unit normal vector of $\bO$ and the stress tensor $\sigma$ defining the stress-strain relation in $\Omega$ is defined by
	\begin{equation}\label{def_E_sigma}
		\sigma(u)  := \lambda \; \div{u} I + 2 \mu \; \mE{u} \,, \qquad 
		\mE{u} :=  \frac{1}{2}\left(\nabla u + \nabla u ^T \right) \,,
	\end{equation}
where $I$ is the identity matrix and $\E$ is called the strain tensor.
\end{definition}

It is convenient to homogenize problem \eqref{prob_forward_non_hom} in the following way: Taking a $\Phi$ such that $\Phi \vert_\GD = g_D$, one then seeks $u := \ut - \Phi$ such that 
  \begin{equation} \label{prob_forward_hom}
  \begin{split}
      -\div{\sigma(u)} & = f + \div{\sigma(\Phi)}, \quad \text{in} \; \Omega \,, \\
      u \,|_{\GD} & = 0 \,, \\
      \sigma(u) \n \, |_{\GT} & = g_T - \sigma(\Phi) \n \, \vert_\GT \,.
  \end{split}
  \end{equation}
Throughout this paper, we make the following
\begin{assumption}\label{ass_main}
Let $f \in \HmoO^N$, $\gD \in \HohbD^N$, and $\gT \in \HmohObT^N$. Furthermore, let $\Phi \in \HoO^N$ be such that $\Phi \vert_\GD = \gD$.
\end{assumption}  

Since we want to consider weak solutions of \eqref{prob_forward_hom}, we make the following

\begin{definition}
Let Assumption~\ref{ass_main} hold. We define the space
\begin{equation*} 
		V := \HoS^N \,, \qquad \text{where} \qquad \HoS : = \{ u \in H^1(\Omega) \vl u |_\GD  = 0 \} \,,
	\end{equation*}
the linear form
	\begin{equation}\label{def_l_lin}
		l(v) := \sprf{f,v} + \sprgT{\gT,v} \,,
	\end{equation}
and the bilinear form
	\begin{equation}\label{def_a_bilin}
		\alm(u,v) := \intV{\left(\lambda \,\div{u}\div{v} + 2\mu \,\mE{u}:\mE{v}\right) } \,,
	\end{equation}
where the expression $\mE{u}:\mE{v}$ denotes the Frobenius product of the matrices $\mE{u}$ and $\mE{v}$, which also induces the Frobenius norm $\norm{\mE{u}}_F := \sqrt{\mE{u}:\mE{u}}$.
\end{definition}

Note that both $\alm(u,v)$ and $l(v)$ are also well defined for $u,v \in \HoO^N$.

\begin{definition}\label{def_weak_sol}
A function $u \in V$ satisfying the variational problem
	\begin{equation}\label{prob_main_weak}
		\alm(u,v) =  l(v) - \alm(\Phi,v)\,, \qquad \forall \, v \in V \,,
	\end{equation}
is called a weak solution of the linearized elasticity problem \eqref{prob_forward_hom}.
\end{definition}

From now on, we only consider weak solutions of \eqref{prob_forward_hom} in the sense of Definition~\ref{def_weak_sol}.
	
\begin{definition}
The set $\Mmb$ of admissible Lam\'e parameters is defined by
	\begin{equation*}
		\Mmb := \left\{ (\lm) \in \LiO^2 \, \vert \, \exists  \, 0 < \eps \leq  \frac{\mb \, \cK^2}{N+2\cK^2} \, : \, \lambda \geq - \eps \,, \, \mu \geq \mb - \eps > 0  \right\} \,.
	\end{equation*}
\end{definition}

Concerning existence and uniqueness of weak solutions, by standard arguments of elliptic differential equations we get the following

\begin{theorem}\label{thm_weak_sol}
Let the Assumption~\ref{ass_main} hold and assume that the Lam\'e parameters $(\lm) \in \Mmb$ for some $\mb > 0$. Then there exists a unique weak solution $u \in V$ of \eqref{prob_forward_hom}. Moreover, there exists a constant $\cLM> 0$ such that
	\begin{equation*}
	\begin{split}
		\norm{u}_\HoO \leq \cLM \Big(\norm{f}_\HmoO  + \cT \norm{\gT}_\HmohObT
		+ \left(N \norm{\lambda}_\LiO + 2 \norm{\mu}_\LiO\right) \norm{\Phi}_\HoO \Big) \,,
	\end{split}
	\end{equation*}
where $\cT$ denotes the constant of the trace inequality \eqref{ineq_trace}.
\end{theorem}
\begin{proof}
This standard result can for example be found in \cite{Valent_2013}. For the constant $\cLM$ one gets $\cLM = (1+\cF^2)/(\mb \, \cK^2)$, where $\cF$ and $\cK$ are the constants of Friedrich's inequality \eqref{ineq_Friedrich} and Korn's inequality \eqref{ineq_Korn}, respectively.
\end{proof}

\section{The Inverse Problem}

After considering the forward problem of linearized elasticity, we now turn to the inverse problem, 
which is to estimate the Lam\'e parameters $\lm$ by measurements of the displacement field $u$. 
More precisely, we are facing the following inverse problem of quantitative elastography:
\begin{problem}\label{prob_main_informal}
	Let Assumption~\ref{ass_main} hold and let $\ud \in \LtO^N$ be a measurement of the true displacement field $u$ satisfying
	\begin{equation}\label{ass_noise}
		\norm{u - \ud}_\LtO \leq \delta \,,
	\end{equation}
	where $\delta \geq 0$ is the noise level. Given the model of linearized elasticity \eqref{prob_forward_non_hom} in the weak form \eqref{prob_main_weak}, the problem is to find the Lam\'e parameters $\lm$.
\end{problem}

The problem of linearized elastography can be formulated as the solution of the operator equation \eqref{prob_main} with the operator
   \begin{equation}\label{def_F}
	\begin{split}
      F: \D(F) := \left\{ (\lambda, \mu ) \in \LiO^2  \vl \lambda \geq 0 \,, \,\mu \geq \mb > 0 \right\} & \to \LtO^N \,,
      \\
      (\lambda,\mu) & \mapsto u(\lambda,\mu) \,,
	\end{split}
  \end{equation}
where $u(\lambda,\mu)$ is the solution of \eqref{prob_main_weak} and hence, we can apply all results from classical inverse problems theory \cite{Engl_Hanke_Neubauer_1996}, given that 
the necessary requirements on $F$ hold. For showing them, it is necessary to write $F$ in a different way: We define the space
	\begin{equation}\label{def_dual}
		\Vs := \left(\HoS^N\right)^* \,,
	\end{equation} 
which is the dual space of $V = \HoS^N$. Next, we introduce the operator $\Atlm$ connected to the bilinear form $\alm$, defined by
	\begin{equation}\label{def_At}
	\begin{split}
		\Atlm : \, \HoO^N & \to \Vs \,,
		\\
		\vt &\mapsto \left( v \mapsto \alm(\vt,v) \right) \,,
	\end{split}
	\end{equation}
and its restriction to $V$, i.e., $A := \At \vert_V$, namely
	\begin{equation}\label{def_A}
	\begin{split}
		\Alm : \, V & \to \Vs \,,
		\\
		v &\mapsto \left( \vb \mapsto \alm(v,\vb) \right) \,.
	\end{split}
	\end{equation}
Furthermore, for $v \in V$ and $\vs \in \Vs$, we define the canonical dual
	\begin{equation*}
		\sprD{\vs,v} = \spr{v,\vs}_{V,\Vs} := \vs(v) \,.
	\end{equation*}
Next, we collect some important properties of $\Atlm$ and $\Alm$. For ease of notation,
	\begin{equation}
		\norm{(\lmb)-(\lm)}_\infty :=  N \norm{\lb - \lambda}_\LiO + 2 \norm{\mub-\mu}_\LiO \,.
	\end{equation}

\begin{proposition}\label{prop_A_general}
The operators $\Atlm$ and $\Alm$ defined by \eqref{def_At} and \eqref{def_A}, respectively, are bounded and linear for all $\lm \in \LiO$. In particular, for all $\lm,\lmb \in \LiO$ 
	\begin{equation}\label{eq_A_bounded}
		\normDDs{\Almb- \Alm} \leq \norm{\Atlmb-\Atlm}_{\HoO,V^*} \leq \norm{(\lmb)-(\lm)}_\infty \,. 
	\end{equation}
Furthermore, for all $(\lm) \in \Mmb$ with $\mb > 0$, the operator $\Alm$ is bijective and has a continuous inverse $\Alm^{-1} : \Vs \to V$ satisfying $\normDsD{\Alm^{-1}} \leq \cLM$, where $\cLM$ is the constant of Theorem~\ref{thm_weak_sol}. In particular, for all $\vs, \vsb \in \Vs$ and $(\lm),(\lmb) \in \Mmb$
	\begin{equation}\label{eq_A_inv_ineq}
	\begin{split}
		\norm{\Almb^{-1} \vsb - \Alm^{-1} \vs}_V \leq
		\cLM \left(\norm{(\lmb)-(\lm)}_\infty\norm{\Alm^{-1} \vs}_V + \norm{\vsb - \vs}_\Vs\right) \,.
	\end{split}
	\end{equation}
\end{proposition}

\begin{proof}
The boundedness and linearity of $\Alm$ and $\Atlm$ for all $\lm \in \LiO$ are immediate consequences of the boundedness 
and bilinearity of $\alm$ and we have 
	\begin{equation*}
	\begin{split}
		&\normHDs{\Atlm - \Atlmb} 
		= 
		\normHDs{\Atlmbd}
		=
		\sup_{u \in \HoO, u \neq 0} \frac{\norm{\Atlmbd u}_\Vs}{\norm{u}_\HoO} 
		\\
		& \quad	= 
		\sup_{u \in \HoO, u \neq 0} \frac{\sup_{v \in V, v \neq 0} \abs{\almbd(u,v)}}{\norm{u}_\HoO \norm{v}_V} 
		\leq
		\norm{(\lmb)-(\lm)}_\infty \,,
	\end{split}
	\end{equation*}
which also translates to $\Alm$, since $V \subset \HoO^N$. Moreover, due to the Lax-Milgram Lemma and Theorem~\ref{thm_weak_sol}, $\Alm$ is bijective for $(\lm) \in \Mmb$ with $\mb > 0$ and therefore, by the Open Mapping Theorem, $\Alm^{-1}$ exists and is linear and continuous. Again by the Lax-Milgram Lemma, there follows $\normDsD{\Alm^{-1}} \leq \cLM$.

Let $\vs, \vsb \in \Vs$ and $(\lm), (\lmb) \in \Mmb$ with $\mb > 0$ be arbitrary but fixed and consider $u := \Alm^{-1} \vs$ and $\ub := \Almb^{-1} \vsb$. Subtracting those two equations, we get
	\begin{equation*}
		\Alm u - \Almb \ub = \vs - \vsb \,,
	\end{equation*}
which, by the definition of $\Alm$ and $\alm$, can be written as
	\begin{equation*}
		 \Almb \left( u - \ub \right)  = \Almbd u + \vs - \vsb \,.
	\end{equation*}
and is equivalent to the variational problem
	\begin{equation}\label{prob_var_continuity}
	\begin{split}
		 \almb\left( \left( u - \ub \right),v\right)  
		 = \almbd (u,v) + \sprD{\vs - \vsb,v} \,, \qquad
		 \forall \, v \in V \,.
	\end{split}
	\end{equation}
Now since $\alm$ is bounded, the right hand side of \eqref{prob_var_continuity} is bounded by
	\begin{equation*}
		\kl{\norm{(\lmb)-(\lm)}_\infty \norm{u}_\HoO   + \norm{\vs - \vsb}_\Vs}\norm{v}_V \,.
	\end{equation*}
Hence, due to the Lax-Milgram Lemma the solution of \eqref{prob_var_continuity} is unique and depends continuously on the right hand side, which immediately yields the assertion.
\end{proof}

Using $\Alm$ and $\Atlm$, the operator $F$ can be written in the alternative form
	\begin{equation}\label{def_F_A}
		F(\lm) = \Alm^{-1} \left( l - \Atlm \Phi \right) \,,
	\end{equation}
with $l$ defined by \eqref{def_l_lin}. Now since, due to \eqref{eq_A_bounded},
	\begin{equation*}
	\begin{split}
		\norm{\left(l - \Atlm \Phi\right) - \left(l - \Atlmb \Phi\right)}_\Vs  =
		\norm{\Atlmbd \Phi}_\Vs 
		\leq
		\norm{(\lmb)-(\lm)}_\infty \norm{\Phi}_\HoO \,,
	\end{split}
	\end{equation*} 
inequality \eqref{eq_A_inv_ineq} implies 
	\begin{equation}\label{ineq_u_cont}
	\begin{split}
		\norm{F(\lmb) - F(\lm)}_V \leq
		\cLM \norm{(\lmb)-(\lm)}_\infty\left(  \norm{F(\lm)}_\HoO   +  \norm{\Phi}_\HoO \right) \,,
	\end{split}
	\end{equation}
showing that $F$ is a continuous operator. 

\begin{remark}
Note that $F$ can also be considered as an operator from $\Mmb$ to $\LtO^N$, in which case Theorem~\ref{thm_weak_sol} and Proposition~\ref{prop_A_general} guarantee that it remains well-defined and continuous, which we use later on. 
\end{remark}

\subsection{Calculation of the Fr\'echet Derivative}\label{sec:frechet}

In this section, we compute the Fr\'echet derivative $F'(\lm)(\hh)$ of $F$ using the representation \eqref{def_F_A}. 
\begin{theorem}\label{thm_F_D}
The operator $F$ defined by \eqref{def_F_A} and considered as an operator from $\Mmb \to \LtO^N$ for some $\mb > 0$ is Fr\'echet differentiable for all $(\lm) \in \D(F)$ with
	\begin{equation}\label{def_F_D}
		F'(\lm)(\hh) = -\Alm^{-1}\left(\Ahh \ulm + \Athh \Phi \right) \,.
	\end{equation}
\end{theorem}
\begin{proof}
We start by defining
	\begin{equation*}
		\Glm(\hh) := -\Alm^{-1}\left(\Ahh \ulm + \Athh \Phi \right) \,.
	\end{equation*}
Due to Proposition~\ref{prop_A_general}, $\Glm$ is a well-defined, bounded linear operator which depends continuously on 
$(\lm) \in D(F)$ with respect to the operator-norm. Hence, if we can prove that $\Glm$ is the Gate\^{a}ux derivative of $F$ it is also the Fr\'echet derivative of $F$. For this, we look at 
	\begin{equation}\label{eq_DF_helper_3}
	\begin{split}
		& \frac{F(\lmth)- F(\lm)}{t} - \Glm(\hh) 
		\\
		& \quad
		=
		\fot \left(\Almth^{-1}(l - \Atlmth \Phi) - \Alm^{-1} (l - \Atlm \Phi)\right) 
		\\
		& \qquad
		+  \Alm^{-1}\left(\Ahh \ulm + \Athh \Phi \right) \,.
	\end{split}
	\end{equation}
Note that it can happen that $(\lmth) \notin \D(F)$. However, choosing $t$ small enough, one can always guarantee that $(\lmth) \in \Mmb$, in which case $F(\lmth)$ remains well-defined as noted above. Applying $\Alm$ to \eqref{eq_DF_helper_3} we get
	\begin{equation*}
	\begin{split}
		& \Alm \kl{\frac{F(\lmth)- F(\lm)}{t} -  \Glm(\hh)} 
		\\
		& 
		=
		\fot \left(\Alm \Almth^{-1}(l - \Atlmth \Phi) - (l - \Atlm \Phi)\right) 
		 + 
		\left(\Ahh \ulm + \Athh \Phi \right) \,,	
	\end{split}
	\end{equation*}
which, together with
	\begin{equation*}
	\begin{split}
		& \Alm \Almth^{-1}(l - \Atlmth \Phi) 	
		\\
		& = 
		(l - \Atlmth \Phi)	- t \Ahh  \Almth^{-1}(l - \Atlmth \Phi) \,,
	\end{split}
	\end{equation*}
yields
	\begin{equation}\label{eq_DF_helper_2}
	\begin{split}
		& \Alm \kl{\frac{F(\lmth)- F(\lm)}{t} - \Glm(\hh)}
		\\
		& 
		=
		- \Ahh \Almth^{-1}(l - \Atlmth \Phi) + \Ahh \ulm 
		\\
		&  
		= - \Ahh \left( \ulmth - \ulm \right) \,.	
	\end{split}
	\end{equation}
By the continuity of $\Alm$ and $\Alm^{-1}$  and due to \eqref{ineq_u_cont} we can deduce that $\Glm$ is indeed the Gate\^{a}ux derivative and, due to the continuous dependence on $(\lm)$, also the Fr\'echet derivative of $F$, which 
concludes the proof. 
\end{proof}

Concerning the calculation of $F'(\lm)(\hh)$, note that it can be carried out in two distinct steps, requiring the solution of two variational problems involving the same bilinear form $\alm$ (which can be used for efficient implementation) as follows:
\begin{enumerate}
    
    \item Calculate $u \in V$ as the solution of the variational problem \eqref{prob_main_weak}.
    
	\item Calculate $F'(\lm)(\hh) \in V$ as the solution $\uh$ of the variational problem
	\begin{equation*}
		\alm(\uh,v) =
		 -\ahh(u,v) -\ahh(\Phi,v) \,,
		\qquad \forall \, v \in V \,.
	\end{equation*}	
           
\end{enumerate}

\begin{remark}
Note that for classical results on iterative regularization methods (see \cite{Kaltenbacher_Neubauer_Scherzer_2008}) 
to be applicable, one needs that both the definition space and the image space are Hilbert spaces. 
However, the operator $F$ given by \eqref{def_F} is defined on $\LiO^2$. Therefore, one could think of applying Banach space regularization theory to the problem (see for example \cite{Schoepfer_Louis_Schuster_2006, Kaltenbacher_Schoepfer_Schuster_2009, Schuster_Kaltenbacher_Hofmann_Kazimierski_2012}). Unfortunately, a commonly used assumption is that the involved Banach spaces are reflexive, which excludes $\LiO^2$. Hence, a commonly used approach is to consider a space which embeds compactly into $\LiO^2$, for example the Banach space $W^{1,p}(\Omega)^2$ or the Hilbert space $\HsO^2$ with $p$ and $s$ large enough, respectively. Although it is preferable to assume as little smoothness as possible for the Lam\'e parameters, we focus on the $\HsO^2$ setting in this paper, since the resulting inverse problem is already difficult enough to treat analytically.  
\end{remark}

Due to Sobolev's embedding theorem \cite{Adams_Fournier_2003}, the Sobolev space $\HsO$ embeds compactly into $\LiO$ for $s > N/2$, i.e., there exists a constant $\cEs > 0$ such that
	\begin{equation}\label{HsO_embedding}
		\norm{v}_\LiO \leq \cEs \norm{v}_\HsO \,, \qquad \forall \, v \in \HsO \,.
	\end{equation}  
This suggests to consider $F$ as an operator from 
	\begin{equation}\label{def_Ds}
	\begin{split}
		\Ds(F) := \{(\lm) \in \HsO^2 \vl \lambda \geq 0 \,, \mu \geq \mb > 0\} &\to \LtO^N \,,
	\end{split}
	\end{equation}
for some $s>N/2$. Since due to \eqref{HsO_embedding} there holds $\Ds(F) \subset \D(F)$, our previous results on continuity and Fr\'echet differentiability still hold in this case. Furthermore, it is now possible to consider the resulting inverse problem $F(\lm) = u$ in the classical Hilbert space framework. Hence, in what follows, we always consider $F$ as an operator from $\Ds(F) \to \LtO^2$ for some $s>N/2$.

\subsection{Calculation of the Adjoint of the Fr\'echet Derivative}\label{sec:frechet_adjoint}
We now turn to the calculation of $F'(\lm)^*w$, the adjoint of the Fr\'echet derivative $F'(\lm)$, which is 
required below for the implementation gradient descent methods. For doing so, note first that for $\Alm$ defined by \eqref{def_A}
	\begin{equation}\label{eq_A_selfadj}
		\sprD{\Alm v,\vb} = \sprD{\Alm \vb,v} \,, \qquad \forall \, v, \vb \in V \,.
	\end{equation}
This follows immediately from the definition of $\Alm$ and the symmetry of the bilinear form $\alm$. 
Moreover, as an immediate consequence of \eqref{eq_A_selfadj}, and continuity of $\Alm^{-1}$ it follows
	\begin{equation}\label{eq_Am_selfadj}
		\sprD{\vs,\Alm^{-1} \vsb } = \sprD{\vsb, \Alm^{-1} \vs } \,, \qquad \forall \, \vs, \vsb \in \Vs \,.                   
	\end{equation}
In order to give an explicit form of $F'(\lm)^*w$ we need the following
\begin{lemma}
The linear operators $T : \LtO^N  \to \Vs $, defined by
	\begin{equation}\label{def_T}
		T w := \left(v \mapsto \intV{w \cdot v}\right) \,,
	\end{equation}
and $\Es : \LoO \to \HsO$,  
	\begin{equation}\label{def_E}
		\spr{\Es u, v}_\HsO = \intV{u v} \,, \qquad \forall v \in \HsO \,,
	\end{equation}
respectively, are well-defined and bounded for all $s>N/2$.
\end{lemma}
\begin{proof}
Using the Cauchy-Schwarz inequality it is easy to see that $T$ is bounded with $\norm{T}_{\LtO,\Vs} \leq 1$. Furthermore, due to \eqref{HsO_embedding},
	\begin{equation*}
		\intV{u v} \leq \norm{u}_\LoO \norm{v}_\LiO \leq \cEs \norm{u}_\LoO \norm{v}_\HsO
		\,, \quad \forall v \in \HsO \,.
	\end{equation*}	
Hence, it follows from the Lax-Milgram Lemma that $\Es$ is bounded for $s > N/2$.
\end{proof}

Using this, we can now proof the main result of this section.

\begin{theorem}\label{thm_F_D_adj}
Let $F :\Ds(F) \to \LtO^2$ with $\Ds(F)$ given as in \eqref{def_Ds} for some $s > N/2$. Then the adjoint of the Fr\'echet derivative of $F$ is given by
	\begin{equation}\label{eq_deriv_adj}
		F'(\lm)^*w = 
		\begin{pmatrix}
		\Es\left(\div{u(\lm) + \Phi} \div{- \Alm^{-1} Tw} \right)
		\\
		\Es\left(2\,\mE{u(\lm) + \Phi} : \mE{- \Alm^{-1} Tw}   \right)
		\end{pmatrix}^T\,,
	\end{equation}
where $T$ and $\Es$ are defined by \eqref{def_T} and \eqref{def_E}, respectively.
\end{theorem}
\begin{proof}
Using Theorem~\ref{thm_F_D} and \eqref{def_T} we get
	\begin{equation*}
	\begin{split}
		&\spr{F'(\lm)(\hh), w}_{\LtO} =
		\spr{- \Alm^{-1} ( \Ahh u(\lm) + \Athh \Phi ), w}_{\LtO}
		\\
		& \qquad =
		\sprD{Tw,- \Alm^{-1} ( \Ahh u(\lm) + \Athh )\Phi )}
	\end{split}
	\end{equation*}
Together with \eqref{eq_Am_selfadj} and the definition of $\Ahh$ and $\ahh$ we get
	\begin{equation*}
	\begin{split}
		&\sprD{Tw,- \Alm^{-1} ( \Ahh u(\lm) + \Athh )\Phi )}
		=
		\ahh\left(u(\lm) + \Phi, - \Alm^{-1} Tw\right)
		\\
		&  =
		\intV{h_\lambda \div{u(\lm) + \Phi} \div{- \Alm^{-1} Tw} }
		+ \intV{2 h_\mu \, \mE{u(\lm) + \Phi} : \mE{- \Alm^{-1} Tw}}\,.
	\end{split}
	\end{equation*}
Together with the fact that the product of two $\LtO$ functions is in $\LoO$, which applies to $\div{u(\lm) + \Phi} \div{- \Alm^{-1} Tw}$ and $ \mE{u(\lm) + \Phi} : \mE{- \Alm^{-1} Tw}$, the statement of the theorem now immediately follows from the definition of $\Es$ \eqref{def_E}.
\end{proof}

Concerning the calculation of $F'(\lm)^*w$, note that it can again be carried out in independent steps, namely:

\begin{enumerate}
 
 	\item Calculate $u \in V$ as the solution of the variational problem \eqref{prob_main_weak}.
     
    \item Compute $\Alm^{-1} T w$, i.e., find the solution $u(w) \in V$ of the variational problem 
    \begin{equation*}
    \alm(u(w),v) = \intV{w \cdot v} \,, \qquad \forall \, v \in V \,.
    \end{equation*}

	\item Compute the functions $u_1(w),u_2(w) \in \LoO$ given by
		\begin{equation*}
		\begin{split}
			u_1(w) := \div{u + \Phi} \div{-u(w)}  \,,
			\\
			u_2(w) := 2 \,\mE{u + \Phi} : \mE{- u(w)}   \,.
		\end{split}
		\end{equation*}
		
	\item Calculate the functions $\lh(w) := \Es \, u_1(w)$ and $\muh(w) := \Es \, u_2(w)$ as the solutions of the variational problems
		\begin{equation*}
		\begin{split}
			\spr{\lh(w), v}_\HsO = \intV{u_1(w) \, v} \,, \qquad \forall v \in \HsO \,,
			\\
			\spr{\muh(w), v}_\HsO = \intV{u_2(w) \, v} \,, \qquad \forall v \in \HsO \,.
		\end{split}
		\end{equation*}

	\item Combine the results to obtain $F'(\lm)^*w = (\lh(w),\muh(w))$.
     
\end{enumerate}

\subsection{Reconstruction of compactly supported Lam\'e parameters}

In many cases, the Lam\'e parameters $\lm$ are known in a small neighbourhood of the boundary, for instance when contact materials are used, such as a gel in ultrasound imaging.  
As a physical problem, we have in mind a test sample consisting of a known material with various inclusions of unknown location and Lam\'e parameters inside. The resulting inverse problem is better behaved than the original problem and we are even able to prove a nonlinearity condition guaranteeing convergence of iterative solution methods for nonlinear ill-posed problems in this case.

More precisely, assume that we are given a bounded, open, connected Lipschitz domain 
$\Oo \subset \Omega$ with  $\bar{\Omega}_1 \Subset \Omega$ and background functions $0 \leq \lambda_b \in \HsO$ and $\mub \leq \mu_b \in \HsO$ and assume that the searched for Lam\'e parameters can be written in the form $(\lambda_b + \lambda, \mu_b + \mu)$, where both $\lm \in \HsO$ are compactly supported in $\Oo$. Hence, after introducing the set
	\begin{equation*}
		\Ds(\Fc) := \left\{ (\lambda, \mu ) \in \HsO^2  \vl \lambda \geq -\lambda_b \,, \,\mu \geq \mb - \mu_b > 0 \,, \, \supp((\lm))\subset \Omega_1 \right\} \,,
	\end{equation*}
we define the operator
   \begin{equation}\label{def_Fc}
	\begin{split}
      \Fc: \Ds(\Fc) \to \LtO^N \,,
      \quad
      (\lm)  \mapsto \Fc(\lm):= F(\lambda_b + \lambda,\mu_b + \mu) \,,
	\end{split}
  \end{equation}
which is well-defined for $s>N/2$. Hence, the sought for Lam\'e parameters can be reconstructed by solving the problem $\Fc(\lm) = u$ and taking $(\lambda_b + \lambda,\mu_b + \mu)$.

Continuity and Fr\'echet differentiability of $F$ also transfer to $\Fc$. For example,
	\begin{equation}\label{def_Fts_D}
		\Fc'(\lm)(\hh) = -A_{(\lambda_b + \lambda,\mu_b + \mu)}^{-1}\left(\Ahh \ulm + \Athh \Phi \right) \,.
	\end{equation}
Furthermore, a similar expression as for the adjoint of the Fr\'echet derivative of $F$ also holds for $\Fc$. Consequently, the computation and implementation of $\Fc$, its derivative and the adjoint can be carried out in the same way as for the operator $F$ and hence, the two require roughly the same amount of computational work. However, as we see in the next section, for the operator $\Fc$ it is possible to prove a nonlinearity condition.

\subsection{Strong Nonlinearity Condition}\label{sec:nc}

The so-called \emph{(strong) tangential cone condition} or \emph{(strong) nonlinearity condition} is the basis of 
the convergence analysis of iterative regularization methods for nonlinear ill-posed problems
\cite{Kaltenbacher_Neubauer_Scherzer_2008}. The nonlinearity condition is a non-standard condition in the 
field of differential equations, because it requires a stability estimate in the image domain of the operator $F$. In the theorem below we show a version of this nonlinearity condition, which is sufficient to prove convergence 
of iterative algorithms for solving \eqref{prob_main}.

\begin{theorem}\label{thm_nonlin_Fts}
Let $F : \Ds(F) \to \LtO^2$ for some $s > N/2 + 1$ and let $\Omega_1 \subset \Omega$ be a bounded, open, connected Lipschitz domain with $\bar{\Omega}_1 \Subset \Omega$. Then for each $(\lm) \in \Ds(F)$  there exists a constant $\cNL = \cNL(\lm,\Oo,\Omega)> 0$ such that for all $(\lmb) \in \Ds(F)$ satisfying $(\lm) = (\lmb)$ on $\Omega\setminus \Oo$ and $(\lm) = (\lmb)$ on $\bO_1$ there holds
	\begin{equation}\label{ineq_nonlin_Fts}
	\begin{split}
		& \norm{F(\lm)-F(\lmb) - F'(\lm)((\lm)-(\lmb))}_\LtO 
		\\
		& \qquad \qquad
		\leq \cNL \norm{(\lmbd)}_\WoiOo \norm{F(\lm)- F(\lmb)}_\LtO \,.
	\end{split}
	\end{equation}
\end{theorem}
\begin{proof}
Let $(\lm),(\lmb) \in \Ds(F)$ with $s > N/2 + 1$ such that $(\lm) = (\lmb)$ on $\Omega\setminus \Oo$ and $(\lm) = (\lmb)$ on $\bO_1$. For the purpose of this proof, set 
$u = F(\lm)$ and $\ub = F(\lmb)$. By definition, we have 
	\begin{equation*}
	\begin{split}
		&\spr{F(\lm)-F(\lmb) - F'(\lm)((\lm)-(\lmb)),w}_\LtO 
		\\
		& \quad
		= \spr{(u - \ub) - \Alm^{-1}\left(\Almbd u + \Atlmbd \Phi \right),w}_\LtO \,.
	\end{split}
	\end{equation*}
Together with \eqref{def_T} and \eqref{eq_Am_selfadj}, we get	
	\begin{equation*}
	\begin{split}
		&\spr{(u - \ub) - \Alm^{-1}\left(\Almbd u + \Atlmbd \Phi \right),w}_\LtO 
		\\
		& \qquad
		= \sprD{\Alm(u - \ub) - \left(\Almbd u + \Atlmbd \Phi \right),\Alm^{-1} Tw} \,,
	\end{split}
	\end{equation*}	
which can be written as
	\begin{equation*}
	\begin{split}
		\sprD{\Almbd\left( \ub - u  \right) ,\Alm^{-1} Tw}
		+ \sprD{\Alm(u - \ub)- \Almbd \ub - \Atlmbd \Phi,\Alm^{-1} Tw}\,.
	\end{split}
	\end{equation*}		
Now since
	\begin{equation*}
	\begin{split}
		&\Alm(u - \ub)- \Almbd \ub - \Atlmbd \Phi
		\\
		& \qquad=				
		l - \Atlm \Phi - \Alm \ub - \Almbd \ub - \Atlmbd \Phi = 0 \,,
	\end{split}
	\end{equation*}
it follows together with \eqref{eq_A_selfadj} that
	\begin{equation*}
	\begin{split}
		&\spr{F(\lm)-F(\lmb) - F'(\lm)((\lm)-(\lmb)),w}_\LtO 
		\\
		& \qquad
		= \sprD{\Almbd\left( \ub - u  \right) ,\Alm^{-1} Tw}
		= \sprD{\Almbd \Alm^{-1} Tw,\ub - u}	\,.
	\end{split}
	\end{equation*}
Introducing the abbreviation $z := \Alm^{-1} Tw$, and using the definition of $\Almbd$
	\begin{equation*}
	\begin{split}
		& \sprD{\Almbd z, \ub - u}  = \almbd( z, \ub - u) 
		\\
		& \quad =
		\intVo{\left( (\lb - \lambda) \,\div{z}\div{ \ub - u} + 2(\mub - \mu ) \,\mE{z}:\mE{ \ub - u}\right) } 	\,,
	\end{split}
	\end{equation*}
where we have used that $(\lmbd) = 0$ on $\Omega\setminus \Oo$. Since we also have $(\lmbd) = 0$ on $\bO_1$, partial integration together with the regularity result Lemma~\ref{lem_regularity} yields
	\begin{equation}\label{helper_nonlin_cond_3}
	\begin{split}
		&\intVo{\left( (\lb - \lambda) \,\div{z}\div{ \ub - u} + 2(\mub - \mu ) \,\mE{z}:\mE{ \ub - u}\right) } 	
		\\
		& \quad = -\intVo{\div{ (\lb - \lambda) \,\div{z}I + 2(\mub - \mu ) \,\mE{z} } \cdot (\ub - u) } 
		\\
		& \quad \leq \norm{ \div{ (\lb - \lambda) \,\div{z}I + 2(\mub - \mu ) \,\mE{z} } }_\LtOo \norm{\ub - u}_\LtOo \,.
	\end{split}
	\end{equation}	
Now, since there exists a constant $\cG = \cG(N)$ such that for all $v \in \HtOo^N$  
	\begin{equation*}
		\norm{ \div{ \lambda \,\div{v}I + 2\mu  \,\mE{v} }  }_\LtOo 
		\leq \cG \max\{\norm{\lambda}_\WoiOo, \norm{\mu}_\WoiOo \} \norm{v}_\HtOo \,.
	\end{equation*}
Now since
	\begin{equation*}
	\begin{split}
		& \norm{F(\lm)-F(\lmb) - F'(\lm)((\lm)-(\lmb))}_\LtO
		\\
		& \quad 
		= \sup_{\norm{w}_\LtO = 1} 
		\spr{F(\lm)-F(\lmb) - F'(\lm)((\lm)-(\lmb)),w}_\LtO \,,
	\end{split}
	\end{equation*}
combining the above results we get
	\begin{equation*}
	\begin{split}
		&\sup_{\norm{w}_\LtO = 1} 
		\spr{F(\lm)-F(\lmb) - F'(\lm)((\lm)-(\lmb)),w}_\LtO 
		\\
		& \quad \leq 
		\sup_{\norm{w}_\LtO = 1} \cG \norm{(\lb -\lambda, \mub - \mu)}_\WoiOo \norm{z}_\HtOo  \norm{\ub - u}_\LtOo  \,.
	\end{split}
	\end{equation*}
Together with Lemma~\ref{lem_regularity}, which implies that there exists a constant $\cR > 0$ such that $\norm{z}_\HtOo \leq \cR \norm{w}_\LtOo$, we get	
	\begin{equation*}
	\begin{split}
		& \norm{F(\lm)-F(\lmb) - F'(\lm)((\lm)-(\lmb))}_\LtO
		\\
		& \quad \leq 
		\cG \, \cR \norm{(\lb -\lambda, \mub - \mu)}_\WoiOo \norm{\ub - u}_\LtOo 
		\\
		& \quad \leq
		\cG \, \cR \norm{(\lb -\lambda, \mub - \mu)}_\WoiOo \norm{\ub - u}_\LtO \,,
	\end{split}
	\end{equation*}
which immediately yields the assertion with $\cNL := \cG \, \cR$.
\end{proof}

We get the following useful corollary

\begin{corollary}\label{corr_nonlin_Fc}
Let $\Fc$ be defined as in \eqref{def_Fc} for some $s > N/2 + 1$. Then for each $(\lm) \in \Ds(\Fc)$  there exists a constant $\cNL = \cNL(\lm,\Oo,\Omega)> 0$ such that for all $(\lmb) \in \Ds(\Fc)$ there holds
	\begin{equation}\label{nonlin_cond_Fc}
	\begin{split}
		& \norm{\Fc(\lm)-\Fc(\lmb) - \Fc'(\lm)((\lm)-(\lmb))}_\LtO 
		\\
		& \qquad \qquad
		\leq \cNL \norm{(\lmbd)}_\WoiOo \norm{\Fc(\lm)- \Fc(\lmb)}_\LtO \,.
	\end{split}
	\end{equation}
\end{corollary}
\begin{proof}
This follows from the definition of $\Fc$ and (the proof of) Theorem~\ref{thm_nonlin_Fts}.
\end{proof}

In the following theorem, we establish a similar result as in Corollary~\ref{corr_nonlin_Fc} now for $F : \Ds(F) \to \LtO^2$ in case that $\GT=\emptyset$, i.e., $\GD = \bO$ and that $\bO$ is smooth enough.

\begin{theorem}\label{thm_nonlin_Fs}
Let $F:\Ds(F) \to \LtO^2$ for some $s > N/2 + 1$ and let $\bO = \GD \in C^{1,1}$ and $\GT = \emptyset$. Then for each $(\lm) \in \Ds(F)$  there exists a constant $\cNL = \cNL(\lm,\Omega)> 0$ such that for all $(\lmb) \in \Ds(F)$, 
there holds
	\begin{equation}\label{ineq_nonlin_Fs}
	\begin{split}
		& \norm{F(\lm)-F(\lmb) - F'(\lm)((\lm)-(\lmb))}_\LtO 
		\\
		& \qquad \qquad
		\leq \cNL \norm{(\lmbd)}_\WoiO \norm{F(\lm)- F(\lmb)}_\LtO \,.
	\end{split}
	\end{equation}
\end{theorem}
\begin{proof}
The prove of this theorem is analogous to the one of Theorem~\ref{thm_nonlin_Fts}, noting that for this choice of boundary condition, the regularity results of Lemma~\ref{lem_regularity} also hold on the entire domain, i.e., for $\Omega_1 = \Omega$, which follows for example from \cite[Theorem~4.16 and Theorem~4.18]{McLean_2000}. Furthermore, the boundary integral appearing in the partial integration step in \eqref{helper_nonlin_cond_3} also vanishes in this case, since $\ub = u = 0$ on $\bO$ due to the assumption that $\bO = \GD$.
\end{proof}

As can be found for example in \cite{Necas_2011, Ciarlet_1994, Agmon_Douglis_Nirenberg_1959, Gilbarg_Trudinger_1998}, $\HtO$ regularity and hence the above theorem can also be proven under weaker smoothness assumptions on the domain $\Omega$. For example, it suffices that $\Omega$ is a convex Lipschitz domain.

\begin{remark}
Note that \eqref{nonlin_cond_Fc} is already strong enough to prove convergence of the Landweber iteration for the operator $\Fc$ to a solution $(\lmD)$ given that the initial guess $(\lmz)$ is chosen close enough to $(\lmD)$ \cite{Hanke_Neubauer_Scherzer_1995, Kaltenbacher_Neubauer_Scherzer_2008}. Furthermore, if there is a $\rb > 0$ such that
	\begin{equation}\label{cond_cR_bounded}
		\sup\limits_{(\lm) \in \, \Brb(\lmD) \cap \Ds(\Fc)} \cR(\lm,\Oo,\Omega) \, < \infty \,,
	\end{equation}
then for each $\eta > 0$ there exists a $\rho > 0$ such that
	\begin{equation*}
	\begin{split}
		\norm{\Fc(\lm)-\Fc(\lmb) - \Fc'(\lm)((\lm)-(\lmb))}_\LtO 
		\leq \eta \norm{\Fc(\lm)-\Fc(\lmb)}_\LtO \,,
		\\
		\forall \, (\lm),(\lmb) \in \Btr(\lmz) \,,
	\end{split}
	\end{equation*}	
which is the original, well-known nonlinearity condition \cite{Hanke_Neubauer_Scherzer_1995}. Obviously, the same statements also hold analogously for the $F:\Ds(F) \to \LtO$ under the assumptions of Theorem~\ref{thm_nonlin_Fs}. 
Note further that condition \eqref{cond_cR_bounded} follows directly from the proofs of \cite[Theorem~4.16 and Theorem~4.18]{McLean_2000}.
\end{remark}

\subsection{An Informal Discussion of Source Conditions}\label{sec:source}

For general inverse problems of the form $F(x) = y$, source conditions of the form
	\begin{equation}\label{sourcecond_general}
		\xD - x_0 \in \Range(F'(\xD)^*) \,,
	\end{equation}
where $\xD$ and $x_0$ denote a solution of $F(x)=y$ and an initial guess, respectively, are important for showing convergence rates or even proving convergence of certain gradient-type methods for nonlinear ill-posed problems \cite{Kaltenbacher_Neubauer_Scherzer_2008}. In this section, we make an investigation of the source condition for $F:\Ds(F) \to \LtO^N$ and $N=2,3$.

\begin{lemma}
Let $F:\Ds(F) \to \LtO^N$ with $s > N/2 + 1$. Then \eqref{sourcecond_general} is equivalent to the existence of a $w \in \LtO^N$ such that
	\begin{equation}\label{sourcecond_Fs_expl}
		\begin{pmatrix}
			\lDz \\ \mDz 
		\end{pmatrix} =
		\begin{pmatrix}
			\Es\left(\div{u(\lmD) + \Phi} \div{- \AlmD^{-1} Tw} \right)
			\\
			\Es\left(2\,\mE{u(\lmD) + \Phi} : \mE{- \AlmD^{-1} Tw}   \right)
		\end{pmatrix}\,.
	\end{equation} 
\end{lemma}
\begin{proof}
This follows immediately from \eqref{thm_F_D_adj}.
\end{proof}

Hence, one has to have that $\lDz \in \Range(\Es)$ and $\mDz \in \Range(\Es)$ and
	\begin{equation}\label{sourcecond_h1}
		\begin{pmatrix}
			\Es^{-1}(\lDz) \\ \Es^{-1}(\mDz) 
		\end{pmatrix} =
		\begin{pmatrix}
			\div{u(\lmD) + \Phi} \div{- \AlmD^{-1} Tw} 
			\\
			2\,\mE{u(\lmD) + \Phi} : \mE{- \AlmD^{-1} Tw}  
		\end{pmatrix}\,.
	\end{equation}
If $\div{u(\lmD) + \Phi} \div{- \AlmD^{-1} Tw}$ and $	2\,\mE{u(\lmD) + \Phi} : \mE{- \AlmD^{-1} Tw} $ are in $\LtO$, which is for example the case if $w$ as well as $f$, $\Phi$, $\gD$ and $\gT$ satisfy additional $L^p(\Omega)$ regularity \cite{Ciarlet_1994}, then $\Es$ coincides with $i^*$, where $i$ is given as the embedding operator from $\HsO \to \LtO$. In this case, $\lDz \in \Range(\Es)$ and $\mDz \in \Range(\Es)$ imply a certain differentiability and boundary conditions on $\lDz$ and $\mDz$. Now, if 
	\begin{equation*}\label{cond_source_div_u}
		\frac{\Es^{-1}(\lDz)}{\div{u(\lmD) + \Phi}} \in \LtO \,,
	\end{equation*}
then \eqref{sourcecond_h1} can be rewritten as
	\begin{equation}\label{sourcecond_h2}
		\begin{pmatrix}
			\Es^{-1}(\lDz)/\div{u(\lmD) + \Phi} \\ \Es^{-1}(\mDz) 
		\end{pmatrix} =
		\begin{pmatrix}
			\div{- \AlmD^{-1} Tw} 
			\\
			2\,\mE{u(\lmD) + \Phi} : \mE{- \AlmD^{-1} Tw}  
		\end{pmatrix}\,.
	\end{equation}	
Since $\AlmD^{-1}Tw \in \, V \subset \HoO^N$, by the Helmholtz decomposition there exists a function $\phi = \phi(w) \in \HtO$ and a vector field $\psi = \psi(w) \in \HtO^N$ such that
	\begin{equation*}
	\begin{split}
		-\AlmD^{-1} Tw = \grad \phi(w) + \grad \times \psi(w) \,,
		\\
		\kl{\grad \phi(w) + \grad \times \psi(w) }\vert_\GD = 0 \,.
	\end{split}
	\end{equation*} 
Hence, \eqref{sourcecond_h2} is equivalent to 
	\begin{equation}\label{sourcecond_h3}
	\begin{split}
		\Delta \phi(w) &= \Es^{-1}(\lDz)/\div{u(\lmD) + \Phi} \,,
		\\
		\Es^{-1}(\mDz) &= 	2\,\mE{u(\lmD) + \Phi} : \mE{ \grad \phi(w) + \grad \times \psi(w) }  \,,
		\\
		\kl{\grad \phi(w) + \grad \times \psi(w) }\vert_\GD &= 0 \,.
	\end{split}
	\end{equation}	
Note that once $\phi$ and $\psi$ are known such that $	-\AlmD^{-1} Tw = \grad \phi + \grad \times \psi$ holds, $w$ can be uniquely recovered in the following way. Due to the Lax-Milgram Lemma, there exists an element $z(\phi,\psi) \in V$ such that
	\begin{equation*}
		-\sprD{\AlmD \kl{\grad \phi + \grad \times \psi},v } = \spr{z(\phi,\psi),v}_V \,, \qquad \forall \, v\in V \,.
	\end{equation*}
However, since
	\begin{equation*}
	\begin{split}
		&-\sprD{\AlmD \kl{\grad \phi(w) + \grad \times \psi(w)},v }  
		\\
		& \qquad = \sprD{Tw,v} = \spr{w,v}_\LtO = \spr{i_V^*w,v}_V \,,
	\end{split}
	\end{equation*}
where $i_V$ denotes the embedding from $V$ to $\LtO^N$, there follows $z(\phi,\psi) \in \Range(i_V^*)$ and $w$ can be recovered by $w = (i_V^*)^{-1}z(\phi,\psi)$.

\begin{remark}
Hence, we derive that the source condition \eqref{sourcecond_Fs_expl} holds for the solution $(\lmD)$ and the initial guess $(\lmz)$ under the following assumptions:
	\begin{itemize}
		\item $\lDz \in \Range(\Es)$ and $\mDz \in \Range(\Es)$ \,,
		\item there holds
		\begin{equation}
			\frac{\Es^{-1}(\lDz)}{\div{u(\lmD) + \Phi}} \in \LtO \,,
		\end{equation}
		\item there exist functions $\phi \in \HtO$ and $\psi \in \HtO^N$ such that
			\begin{equation*}
			\begin{split}
				\Delta \phi &= \Es^{-1}(\lDz)/\div{u(\lmD) + \Phi} \,,
				\\
				2\,\mE{u(\lmD) + \Phi} : \mE{ \grad \phi + \grad \times \psi } &= 	\Es^{-1}(\mDz)  \,,
				\\
				\kl{\grad \phi + \grad \times \psi }\vert_\GD &= 0 \,,
			\end{split}
			\end{equation*}
		\item the unique weak solution $z(\phi,\psi) \in V$ of the variational problem
			\begin{equation*}
				-\sprD{\AlmD \kl{\grad \phi + \grad \times \psi},v } = \spr{z(\phi,\psi),v}_V \,, \qquad \forall \, v\in V \,, 
			\end{equation*}
		satisfies $z(\phi,\psi) \in \Range(i_V^*)$.
	\end{itemize}
\end{remark}

The above assumptions are restrictive, which is as usual \cite{Kaltenbacher_Neubauer_Scherzer_2008}. However, without these 
assumptions one cannot expect convergence rates.

\begin{remark}
Note that since $u(\lmD) + \Phi$ is the weak solution of the non-homogenized problem \eqref{prob_forward_non_hom}, condition \eqref{cond_source_div_u} implies that in areas of a divergence free displacement field, one has to know the true Lam\'e parameter $\lD$. This should be compared to similar conditions in \cite{Widlak_Scherzer_2015,Bal_Uhlmann_2012,Bal_Uhlmann_2013, Bal_Naetar_Scherzer_Scotland_2013}.
\end{remark}

\begin{remark}
Note that if the source condition is satisfied, then it is known that the iteratively regularized Landweber and Gauss-Newton 
iterations converge, even if the nonlinearity condition is not satisfied \cite{Scherzer_1996, Bakuschinsky_Kokurin_2004, Bakushinskii_1992}. 
\end{remark}

\section{Numerical Examples}

In this section, we present some numerical examples demonstrating the reconstructions of Lam\'e parameters from given noisy displacement field measurements $\ud$ using both the operators $\Frs$ and $\Fc$ considered above. The sample problem, described in detail in Section~\ref{sect_setting} is chosen in such a way that it closely mimics a possible real-world setting described below. Furthermore, results are presented showing the reconstruction quality for both smooth and non-smooth Lam\'e parameters.

\subsection{Regularization Approach - Landweber Iteration}

For reconstructing the Lam\'e parameters, we use a Two-Point Gradient (TPG) method \cite{Hubmer_Ramlau_2017} based on Landweber's 
iteration and on Nesterov's acceleration scheme \cite{Nesterov_1983} which, using the abbreviation 
$\xkd = \kl{\lkd,\mu_k^\delta}$, read as follows,
    \begin{equation}\label{Nesterov}
         \begin{split}
            \zkd &= \xkd + \akd \kl{\xkd - \xkmd} 
            \,,
            \\ 
            \xkpd &= \zkd + \okd\kl{\zkd} \skd\kl{\zkd} \,,
            \quad \skd\kl{x} := F'\kl{x}^*\kl{\ud - F\kl{x}} \,.
        \end{split}
    \end{equation} 
For linear ill-posed problems, a constant stepsize $\okd$ and $\akd = (k-1)/(k+\alpha-1)$, this method was analysed in \cite{Neubauer_2017}. For nonlinear problems, convergence of \eqref{Nesterov} under the tangential cone condition was shown in \cite{Hubmer_Ramlau_2017} when the discrepancy principle is used as a stopping rule, i.e., the iteration is stopped after $\ks$ steps, with $\ks$ satisfying
    \begin{equation}\label{discrepancy_principle}
        \norm{\ud - F\kl{\xksd}} \le \tau \delta \le \norm{\ud - F\kl{\xkd}}\,, \qquad 0\le k \le k_*\,,
    \end{equation}
where the parameter $\tau$ should be chosen such that
	\begin{equation*}
		\tau > 2\frac{1+\eta}{1-2\eta},
	\end{equation*} 
although the choices $\tau = 2$ or $\tau$ close to $1$ suggested by the linear case are also very popular. For the stepsize $\okd$ we use the steepest descent stepsize \cite{Scherzer_1996} and for $\akd$ we use the well-known Nesterov choice, i.e.,
    \begin{equation}\label{steepest_descent}
        \okd(x) := \frac{\norm{\skd\kl{x}}^2 }{\norm{F'(x)\skd(x)}^2} \,,
        \qquad \text{and} \qquad
        \akd = \frac{k-1}{k+2} \,.
    \end{equation}
The method \eqref{Nesterov} is known to work well for both linear and nonlinear inverse problems \cite{Jin_2016,Hubmer_Neubauer_Ramlau_Voss_2018} and also serves as the basis of the well-known FISTA algorithm \cite{Beck_Teboulle_2009} for solving linear ill-posed problems with sparsity constraints.

\subsection{Problem Setting, Discretization, and Computation}\label{sect_setting}

A possible real-world problem the authors have in mind is a cylinder shaped object made out of agar with a symmetric, ball shaped inclusion of a different type of agar with different material properties and hence, different Lam\'e parameters. The object is placed on a surface and a constant downward displacement is applied from the top while the outer boundary of the object is allowed to move freely. Due to a marker substance being injected into the object beforehand, the resulting displacement field can be measured inside using a combination of different imaging modalities. Since the object is rotationally symmetric, this also holds for the displacement field, which allows for a relatively high resolution $2$D image.

Motivated by this, we consider the following setup for our numerical example problem: For the domain $\Omega$, we choose a rectangle in $2$D, i.e., $N=2$. We split the boundary $\bO$ of our domain into a part $\GD$ consisting of the top and the bottom edge of the rectangle and into a part $\GT$ consisting of the remaining two edges. Since the object is free to move on the sides, we set a zero traction condition on $\GT$, i.e., $\gT = 0$. Analogously for $\GD$, since the object is fixed to the surface and a constant displacement is being applied from above, we set $\gD = 0$ and $\gD = \cP = \text{const}$ on the parts of $\GD$ corresponding to the bottom and the top edge of the domain. 

If, for simplicity, we set $\Omega = (0,1)^2$, then the underlying non-homogenized forward problem \eqref{prob_forward_non_hom} simplifies to
  \begin{alignat}{2} 
      -\div{\sigma(\ut(x))} & = 0 \,, \quad && x \in (0,1)^2 \,, \nonumber\\
      \ut(x)  & = 0 \,, \quad && x \in [0,1]\times \{0\} \,,\nonumber\\
      \ut(x)  & = \cP \,, \quad && x \in  [0,1]\times \{1\} \,, \nonumber\\
      \sigma(\ut(x)) \n(x)  & = 0 \,, \quad && x \in {\{0,1\}\times[0,1]} \,. \label{numerics_prob}
  \end{alignat}
The homogenization function $\Phi$ can be chosen as $\Phi(x_1,x_2) := \cP \, x_2$ in this case.

In order to define the exact Lam\'e parameters $(\lmD)$, we first need to introduce the following family $\Brh$ of symmetric $2$D bump functions with a circular plateau
	\begin{equation*}
		\Brh(x,y) := 
		\begin{cases}
		h_1 \,, & \sqrt{x^2+y^2} \leq r_1 \,, 
		\\
		h_2 \,, & \sqrt{x^2+y^2} \geq r_2 \,,
		\\
		\Srh(\sqrt{x^2 + y^2}) \,, & r_1 < \sqrt{x^2 + y^2} < r_2 \,,
		\end{cases}
	\end{equation*}
where $\Srh$ is a $5$th order polynomial chosen such that the resulting function $\Brh$ is twice continuously differentiable. The exact Lam\'e parameters $(\lmD)$ are then created by shifting the function $\Brh$ and using different values of $r_1,r_2,h_1,h_2$; see Figure~\ref{fig_Lame_smooth}.	
	
\begin{figure}[H]
\centering
\includegraphics[width=0.45\textwidth]{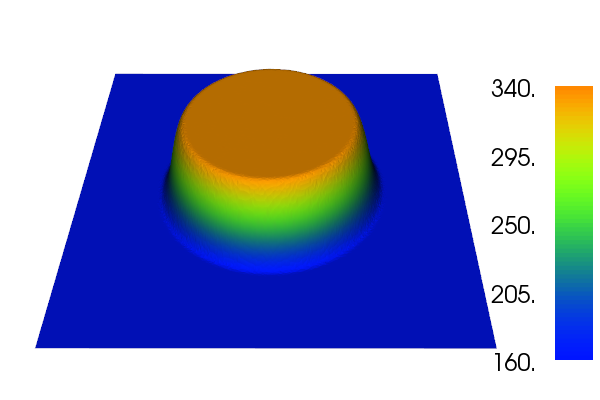}
\includegraphics[width=0.45\textwidth]{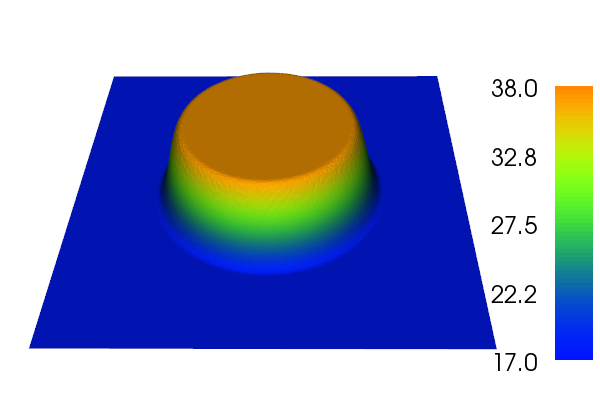}
\caption{Exact Lam\'e parameters $(\lmD)$, in kPa.}
\label{fig_Lame_smooth}
\end{figure} 

As we have seen, a certain smoothness in the exact Lam\'e parameters is required for reconstruction with the operators $F \vert_{\Ds(F)} $ and $\Fc$. Although this might be an unnatural assumption in some cases as different materials next to each other may have Lam\'e parameters of high contrast, it can be justified in the case of the combined agar sample, since when combining the different agar samples into one, the transition from one type of agar into the other can be assumed to be continuous, leading to a smooth behaviour of the Lam\'e parameters in the transition area.

However, since we also want to see the behaviour of the reconstruction algorithm in case of non-smooth Lam\'e parameters $(\lmD)$, we also look at $(\lmD)$ depicted in Figure~\ref{fig_Lame_non_smooth}, which were created using $\Brh$ with $r_1 \approx r_2$ and which, although being twice continuously differentiable in theory, behave like discontinuous functions after discretization.

\begin{figure}[H]
\centering
\includegraphics[width=0.45\textwidth]{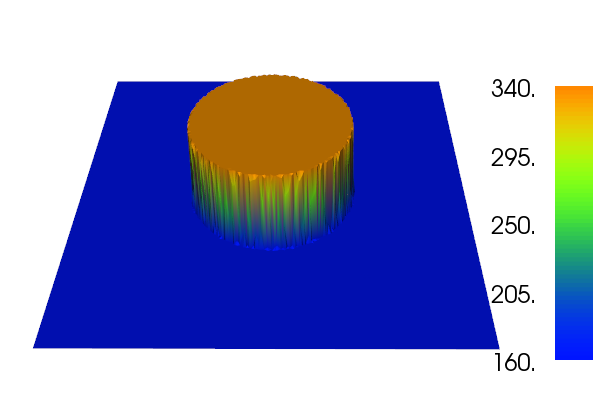}
\includegraphics[width=0.45\textwidth]{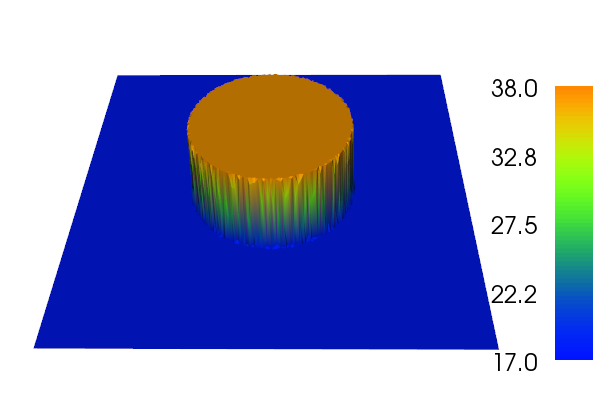}
\caption{Exact Lam\'e parameters $(\lmD)$ created from $\Brh$ with $r_1\approx r_2$, in kPa.}
\label{fig_Lame_non_smooth}
\end{figure} 

The discretization, implementation and computation of the involved variational problems was done using Python and the library FEniCS \cite{Alnaes_Blechta_2015}. For the solution of the inverse problem a triangulation with $4691$ vertices was introduced for discretizing the Lam\'e parameters. The data $u$ was created by applying the forward model \eqref{numerics_prob} to $(\lmD)$ using a finer discretization with $28414$ vertices in order to avoid an inverse crime. For the constant $\cP$ in \eqref{numerics_prob} the choice $\cP = -10^{-4}$ is used. The resulting displacement field for the smooth Lam\'e parameters $(\lmD)$ is depicted in Figure~\ref{fig_Displacement_smooth}. Afterwards, a random noise vector with a relative noise level of $0.5\%$ is added to $u$ to arrive at the noisy data $\ud$. This leads to  the absolute noise level $\delta = \norm{u - \ud}_\LtO \approx 3.1*10^{-7}$. Note that while with a smaller noise level more accurate reconstructions can be obtained, the required computational time then drastically increases  due to the discrepancy principle. Furthermore, a very small noise level is unrealistic in practice.

\begin{figure}[H]
\centering
\includegraphics[width=0.70\textwidth]{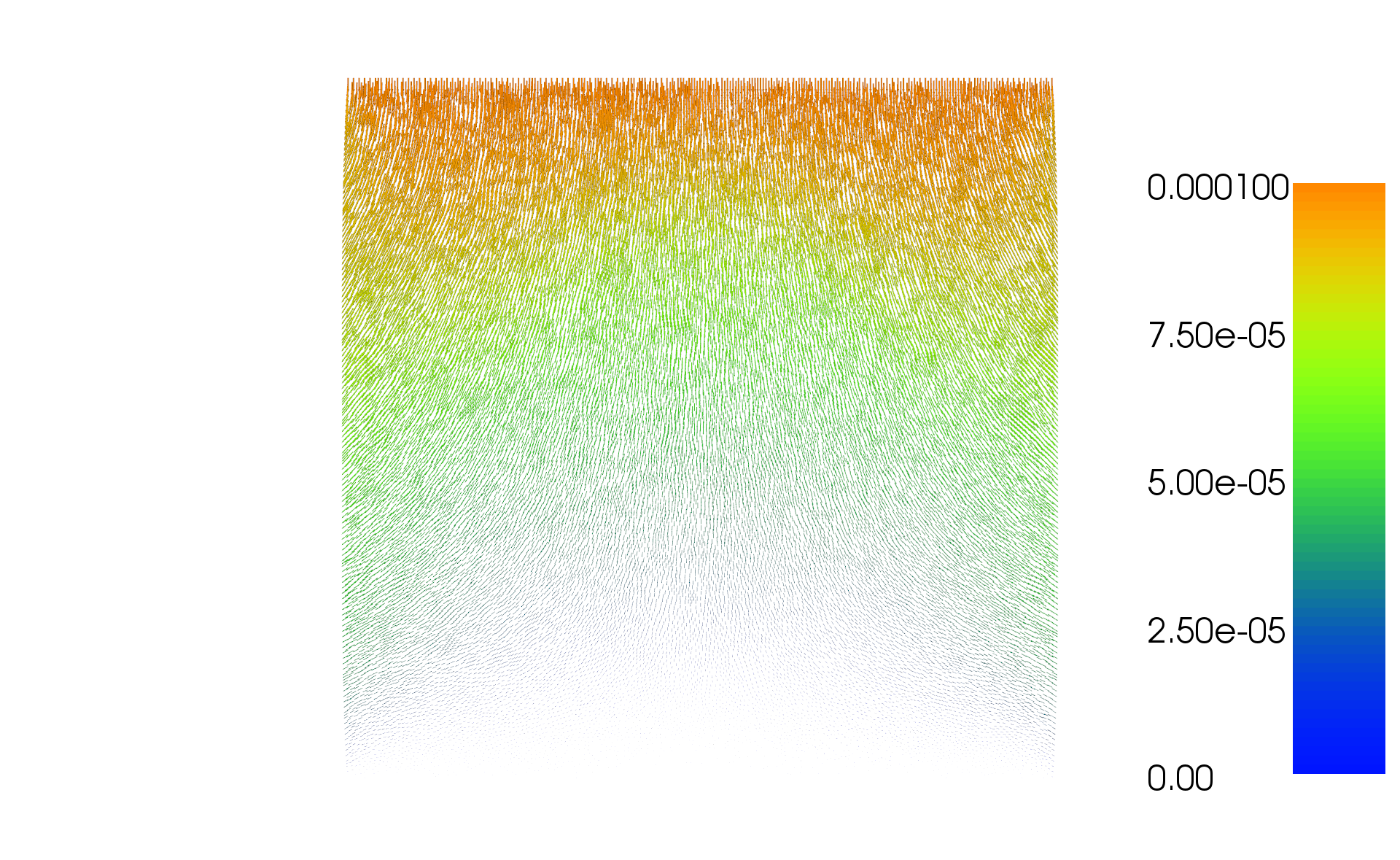}
\caption{Displacement field $u$ corresponding to the Lam\'e parameters $(\lmD)$ depicted in Figure~\ref{fig_Lame_smooth}.}
\label{fig_Displacement_smooth}
\end{figure}

\subsection{Numerical Results}\label{sec:numerics}

In this section we present various reconstruction results for different combinations of operators, Lam\'e parameters and boundary conditions. Since the domain $\Omega$ is two-dimensional, i.e., $N =2$, the operators $F\vert_{\Ds(F)}$ and $\Fc$ are well-defined for any $s > 1$. By our analysis above, we know that the nonlinearity condition holds for the operator $\Fc$ if $s > N/2 + 1$ which suggests to use $s>2$. However, since numerically there is hardly any difference between using $s = 2$ and $s=2+\eps$ for $\eps$ small enough, we choose $s=2$ for ease of implementation in the following examples. When using the operator $\Fc$ we chose a slightly smaller square than $\Omega$ for the domain $\Oo$, which is visible in the reconstructions. Unless noted otherwise, the accelerated Landweber type method \eqref{Nesterov} was used together with the steepest descent stepsize \eqref{steepest_descent} and the iteration was terminated using the discrepancy principle \eqref{discrepancy_principle} together with $\tau = 1$. Concerning the initial guess, when using the operator $\Frs$ the choice $(\lmz) = (2,0.3)$ was made while when using the operator $\Fc$ a zero initial guess was used. For all presented examples, the computation times lay between 15 minutes and 1 hour on a Lenovo ThinkPad W540 with Intel(R) Core(TM) i7-4810MQ CPU @ 2.80GHz, 4 cores.

\begin{example}\label{ex_Fc_main}
As a first test we look at the reconstruction of the smooth Lam\'e parameters (Figure~\ref{fig_Lame_smooth}), using the operator $\Fc$. The iteration terminated after $642$ iterations yields the reconstructions depicted in Figure~\ref{fig_Fswave_01}. The parameter $\mD$ is well reconstructed both qualitatively and quantitatively, with some obvious small artefacts around the border of the inner domain $\Oo$. The parameter $\lD$ is less well reconstructed, which is a common theme throughout this section and is due to the smaller sensitivity of the problem to changes of $\lambda$. However, the location and also quantitative information of the inclusion is obtained.
\begin{figure}[!htb]
\centering
\includegraphics[width=0.45\textwidth]{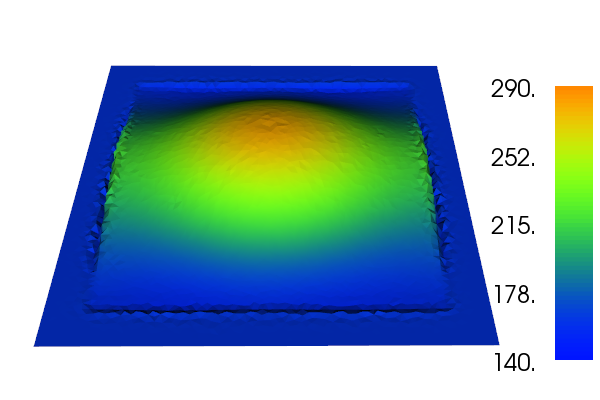}
\includegraphics[width=0.45\textwidth]{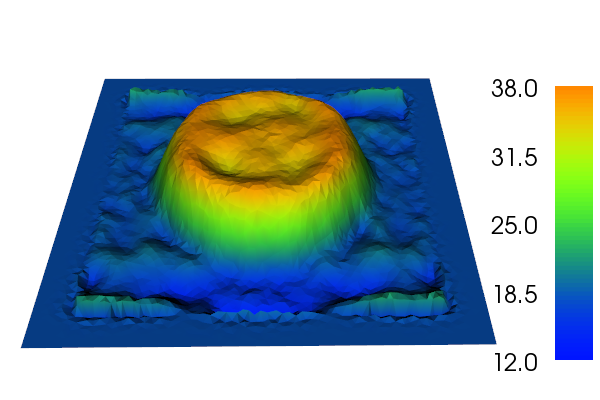}
\caption{Reconstructions of $(\lmD)$, in kPa, Example~\ref{ex_Fc_main}. Smooth Lam\'e parameters (Figure~\ref{fig_Lame_smooth}) - Displacement-Traction boundary conditions - operator $\Fc$.}
\label{fig_Fswave_01}
\end{figure}  
\end{example}

\begin{example}\label{ex_2}
Using the same setup as before, but this time with the operator $\Frs$ instead of $\Fc$ leads to the reconstructions depicted in Figure~\ref{fig_Fs_02}, the discrepancy principle being satisfied after $422$ iterations in this case. Even though information about the Lam\'e parameters can be obtained also here, the reconstructions are worse than in the previous case. 
Note that in the case of mixed boundary conditions the nonlinearity condition has not been verified for the operator $\Frs$, and there is no proven convergence result.
\begin{figure}[!htb]
\centering
\includegraphics[width=0.45\textwidth]{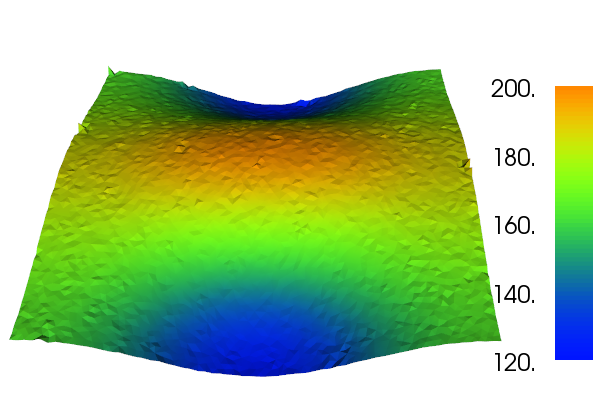}
\includegraphics[width=0.45\textwidth]{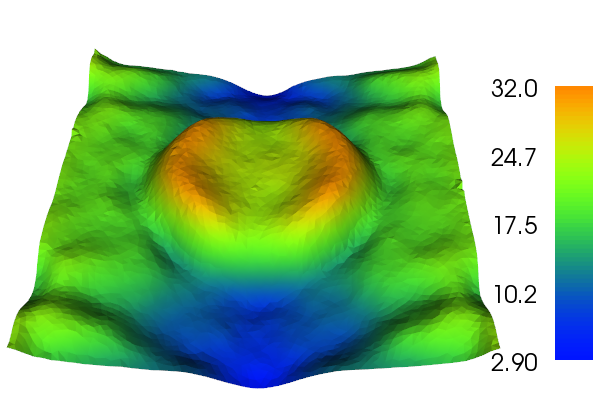}
\caption{Reconstructions of $(\lmD)$, in kPa, Example~\ref{ex_2}. Smooth Lam\'e parameters (Figure~\ref{fig_Lame_smooth}) - Displacement-Traction boundary conditions - operator $\Frs$.}
\label{fig_Fs_02}
\end{figure}  
\end{example}

\begin{example}\label{ex_3}
Going back to the operator $\Fc$ but now using the non-smooth Lam\'e parameters (Figure~\ref{fig_Lame_non_smooth}), we obtain the reconstructions depicted in Figure~\ref{fig_Fswave_3a} after $635$ iterations. We get similar results as for the first test with the main difference that the reconstructed values of the inclusion now fit less well than before, which is due to the non-smoothness of the used Lam\'e parameters.
\begin{figure}[!htb]
\centering
\includegraphics[width=0.45\textwidth]{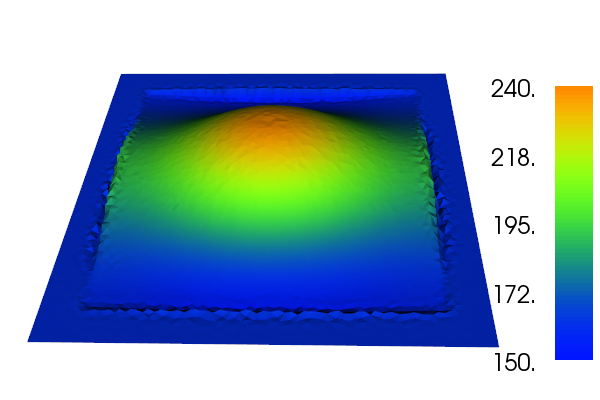}
\includegraphics[width=0.45\textwidth]{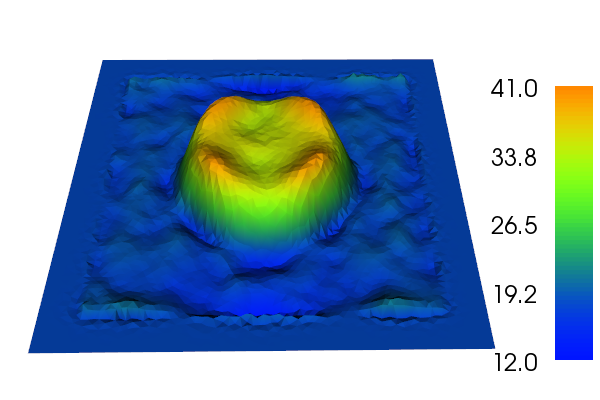}
\caption{Reconstructions of $(\lmD)$, in kPa, Example~\ref{ex_3}. Non-smooth Lam\'e parameters (Figure~\ref{fig_Lame_non_smooth}) - Displacement-Traction boundary conditions - operator $\Fc$.}
\label{fig_Fswave_3a}
\end{figure}  
\end{example}

\begin{example}\label{ex_4}
For the following tests, we want to see what happens if, instead of mixed displacement-traction boundary conditions, only pure displacement conditions are used. For this, we replace the traction boundary condition in \eqref{numerics_prob} by a zero displacement condition while leaving everything else the same. The resulting reconstructions using the operator $\Fc$ for both smooth and non-smooth Lam\'e parameters are depicted in Figures~\ref{fig_Fs_dis_1} and \ref{fig_Fs_dis_2}. The discrepancy principle stopped after $177$ and $194$ iterations, respectively. Compared to the previous tests, it is obvious that the parameter $\lD$ is now much better reconstructed than before in both cases. Also the parameter $\mD$ is well reconstructed, although not as good as in the case of mixed boundary conditions. The influence of the non-smooth Lam\'e parameters in Figure~\ref{fig_Fs_dis_2} can best be seen in the volcano like appearance of the reconstruction of $\mD$.
\begin{figure}[!htb]
\centering
\includegraphics[width=0.45\textwidth]{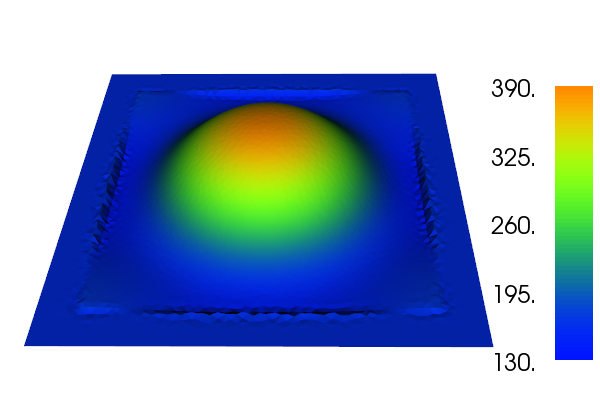}
\includegraphics[width=0.45\textwidth]{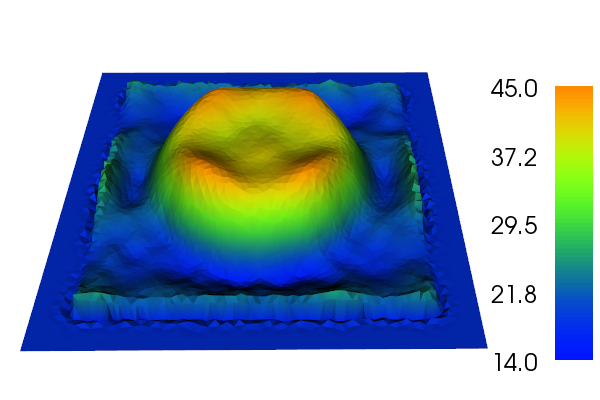}
\caption{Reconstructions of $(\lmD)$, in kPa, Example~\ref{ex_4}. Smooth Lam\'e parameters (Figure~\ref{fig_Lame_smooth}) - Pure displacement boundary conditions - operator $\Fc$.}
\label{fig_Fs_dis_1}
\end{figure}  
\begin{figure}[!htb]
\centering
\includegraphics[width=0.45\textwidth]{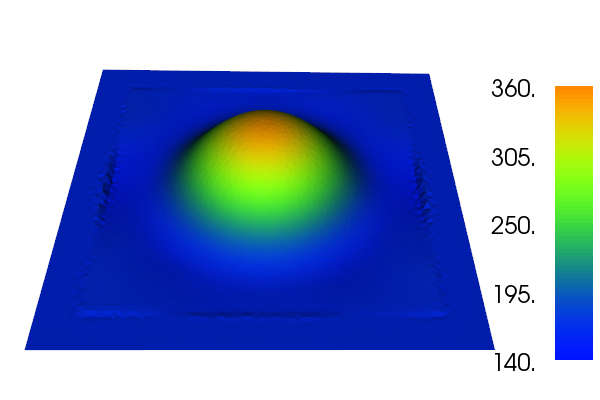}
\includegraphics[width=0.45\textwidth]{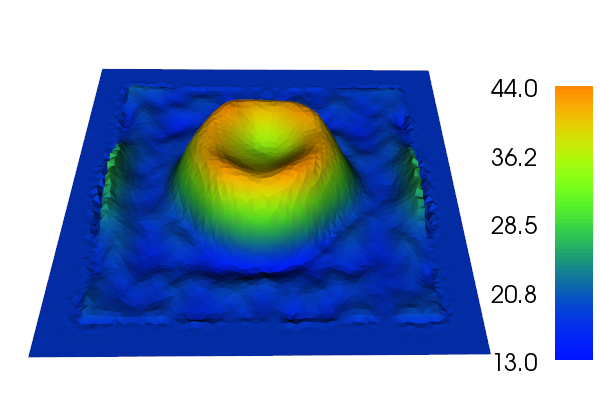}
\caption{Reconstructions of $(\lmD)$, in kPa, Example~\ref{ex_4}. Non-smooth Lam\'e parameters (Figure~\ref{fig_Lame_non_smooth}) - Pure displacement boundary conditions - operator $\Fc$.}
\label{fig_Fs_dis_2}
\end{figure} 
\end{example}

\begin{example}\label{ex_5}
Next, we take a look at the reconstruction of the smooth Lam\'e parameters using $\Frs$ and as before the pure displacement boundary conditions. Interestingly, Nesterov acceleration does not work well in this case and so the Landweber iteration with the steepest descent stepsize was used to obtain the reconstructions depicted in Figure~\ref{fig_Fs_dis_3}, the discrepancy principle being satisfied after $937$ iterations. As with the reconstructions obtained in case of mixed boundary conditions, this case is worse than when using $\Fc$, for the same reasons mentioned above. Note however that in comparison with Figure~\ref{fig_Fs_02}, the inclusion in $\lD$ is much better resolved now than in the other case, which is due to the use of pure displacement boundary conditions.
\begin{figure}[!htb]
\centering
\includegraphics[width=0.45\textwidth]{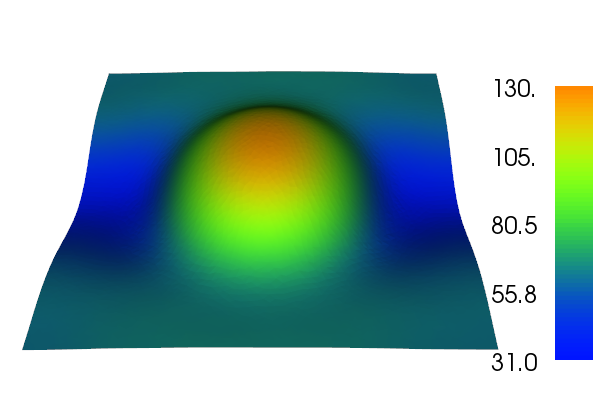}
\includegraphics[width=0.45\textwidth]{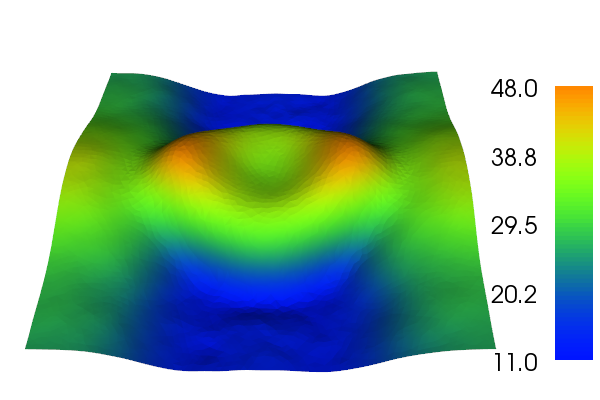}
\caption{Reconstructions of $(\lmD)$, in kPa, Example~\ref{ex_5}. Smooth Lam\'e parameters (Figure~\ref{fig_Lame_smooth}) - Pure displacement boundary conditions - operator $\Frs$.}
\label{fig_Fs_dis_3}
\end{figure}  
\end{example}

\begin{example}\label{ex_6}
For the last test we return to the same setting as in Example~\ref{ex_Fc_main}, i.e., we again use the operator $\Fc$ and mixed displacement-traction boundary conditions. However, this time we consider different exact Lam\'e parameters modelling a material sample with three inclusions of varying elastic behaviour. The exact parameters and the resulting reconstructions, obtained after $921$ iterations, are depicted in Figure~\ref{fig_Fc_dis_trac}. As expected, the Lam\'e parameter $\mD$ is well reconstructed in shape, value and location of the inclusions. Moreover, even though the reconstruction of $\lD$ does not exhibit the same shape as the exact parameter, information about the value and the location of the inclusions was obtained.
\begin{figure}[!htb]
\centering
\includegraphics[width=0.45\textwidth]{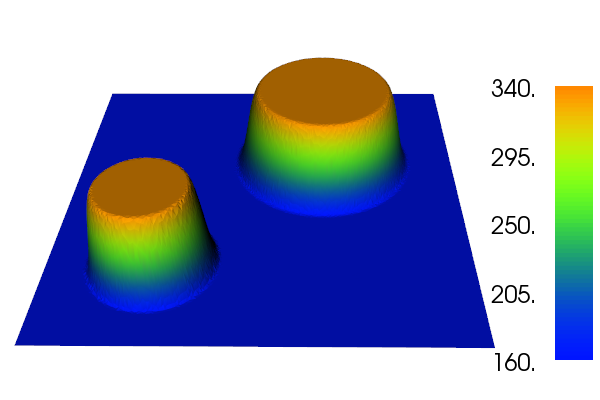}
\includegraphics[width=0.45\textwidth]{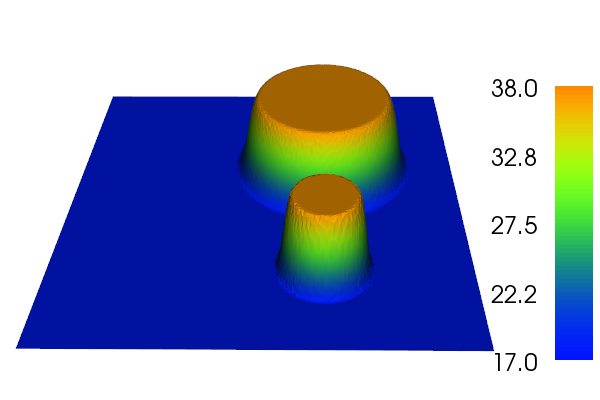}
\\
\includegraphics[width=0.45\textwidth]{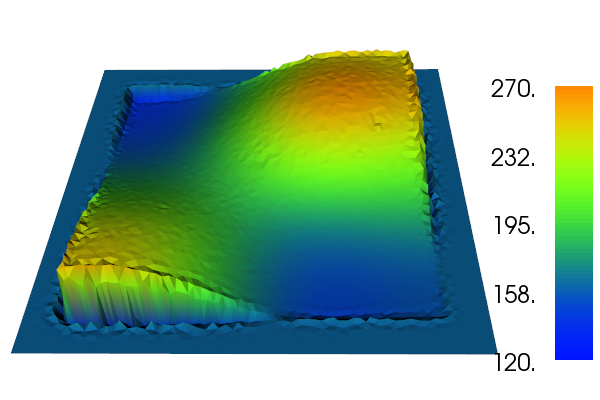}
\includegraphics[width=0.45\textwidth]{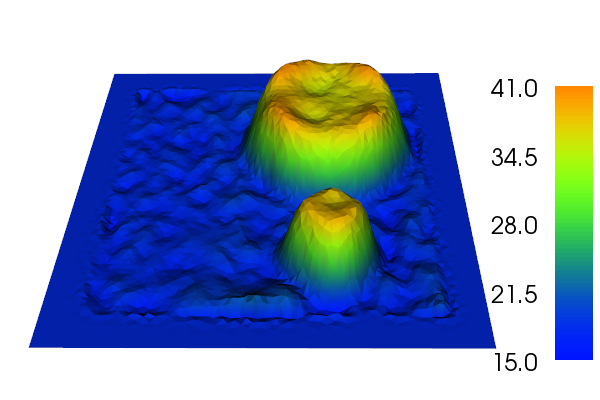}
\caption{Exact Lam\'e parameters $(\lmD)$ (top) and their reconstructions (bottom), in kPa, Example~\ref{ex_6} - Displacement-Traction boundary conditions - operator $\Fc$.}
\label{fig_Fc_dis_trac}
\end{figure}  
\end{example}

\section{Support and Acknowledgements}
The first author was funded by the Austrian Science Fund (FWF): W1214-N15, project DK8. 
The second author was funded by the Danish Council for Independent Research - Natural Sciences: grant 4002-00123. 
The fourth author is also supported by the FWF-project ``Interdisciplinary Coupled Physics Imaging'' (FWF P26687).
The authors would like to thank Dr.\ Stefan Kindermann for providing valuable suggestions and insights during discussions on the subject.

\section*{Appendix. Important results from PDE theory}

Here we collect important results in the theory of partial differential used throughout this paper. Two basic results are the trace inequality \cite{Adams_Fournier_2003}, which states that there exists a constant $\cT = \cT(\Omega) > 0$ such that 
	\begin{equation}\label{ineq_trace}
		\norm{v}_\HohbT \leq \cT \norm{v}_\HoO \,, \quad \forall \, v \in V \,,
	\end{equation}
and Friedrich's inequality \cite{Evans_1998}, i.e., there exists a constant $\cF = \cF(\Omega)>0$ such that
	\begin{equation}\label{ineq_Friedrich}
		\norm{v}_\LtO \leq \cF  \norm{\nabla v}_\LtO \,, \quad \forall \, v\in V \,,		
	\end{equation} 
from which we can deduce
	\begin{equation}\label{ineq_Friedrich_2}
		\norm{v}_\HoO^2 \leq (1+\cF^2) \norm{\nabla v}_\LtO^2 \,, \quad \forall \, v\in V \,.
	\end{equation}
Korn's inequality \cite{Valent_2013} states that there exists a constant $\cK = \cK(\Omega) > 0$ such that
	\begin{equation}\label{ineq_Korn}
		\intV{\norm{\mE{v}}_F^2} \geq \cK^2 \norm{\grad v}_\LtO^2 \,, \qquad \forall \, v \in V \,.
	\end{equation} 
Furthermore, we need the following regularity result
\begin{lemma}\label{lem_regularity}
Let $(\lm) \in \Ds(F)$ with $s>N/2+1$ and $w \in \LtO^N$. Then there exists a unique weak solution $u$ of the elliptic boundary value problem
	\begin{equation} \label{prob_regularity}
	\begin{split}
	    -\div{\sigma(u)} & = w \,, \quad \text{in} \; \Omega \,, \\
	    u \,|_{\GD} & = 0 \,, \\
	    \sigma(u) \n \,|_{\GT} & = 0 \,,
	\end{split}
	\end{equation}
and for every bounded, open, connected Lipschitz domain $\Oo \subset \Omega$ with $\bar{\Omega}_1 \Subset \Omega$ there holds $u \vert_\Oo \in \HtOo^N$ and $-\div{\sigma(u)} = w$ pointwise almost everywhere in $\Oo$. Furthermore, there is a constant $\cR = \cR(\lm,\Oo,\Omega)$ such that 
	\begin{equation}\label{ineq_regularity}
		\norm{u}_\HtOo \leq \cR \norm{w}_\LtOo \,.
	\end{equation}
\end{lemma}
\begin{proof}
This follows immediately from \cite[Theorem~4.16]{McLean_2000}.
\end{proof}


\bibliographystyle{plain}
\bibliography{mybib}

\end{document}